\documentclass[11pt]{article}
\usepackage{amsmath, graphicx, amsfonts,amssymb, calrsfs}
\usepackage{amsfonts,mathrsfs, color, amsthm}
\usepackage{bm}

\usepackage{enumitem}
\usepackage{hyperref}
\hypersetup{
  colorlinks   = true, 
  urlcolor     = blue, 
  linkcolor    = blue, 
  citecolor   = red 
  }

\def\cK{\mathcal{K}}
\def\sphere{S^{n-1}}

\def\Rn{{\mathbb R^n}}
 \def\R{\mathbb{R}}

\newtheorem{problem}{Problem}[section]

\newtheorem{theorem}{Theorem}[section]
\newtheorem{lemma}{Lemma}[section]
\newtheorem{remark}{Remark}[section]
\newtheorem{proposition}{Proposition}[section]
\newtheorem{corollary}{Corollary}[section]

\newtheorem{definition}{Definition}[section]

\def\bt{\begin{theorem}}
\def\et{\end{theorem}}
\def\bl{\begin{lemma}}
\def\el{\end{lemma}}
\def\br{\begin{remark}}
\def\er{\end{remark}}
\def\bc{\begin{corollary}}
\def\ec{\end{corollary}}
\def\bd{\begin{definition}}
\def\ed{\end{definition}}
\def\bp{\begin{proposition}}
\def\ep{\end{proposition}}

\def\copsum{\alpha_1 \circ A_{1} \oplus_{p}\alpha_2 \circ A_{2}}

\begin{document}
\title{On the $L_p$ Brunn-Minkowski theory and the $L_p$ Minkowski problem for $C$-coconvex
sets }

\author{Jin Yang, Deping Ye and Baocheng Zhu}
\date{}
\maketitle

\begin{abstract}
 Let  $C$ be a pointed closed convex cone in $\Rn$
with vertex at the origin $o$ and having nonempty interior.   The
set $A\subset C$ is $C$-coconvex if the volume of $A$ is finite and
$A^{\bullet}=C\setminus A$ is a closed convex set. For $0<p<1$,  the
$p$-co-sum of $C$-coconvex sets is introduced, and the corresponding
$L_p$ Brunn-Minkowski inequality for $C$-coconvex sets is
established.  We also define the $L_p$ surface area measures, for
$0\neq p\in \R$,  of certain $C$-coconvex sets, which are critical
in deriving a variational formula of the volume of the Wulff shape
associated with a family of  functions obtained from the $p$-co-sum.
This motivates the $L_p$ Minkowski problem aiming to characterize
the $L_p$ surface area measures of $C$-coconvex sets. The existence
of solutions to the  $L_p$ Minkowski problem for all $0\neq p\in\R$
is established. The $L_p$ Minkowski inequality for $0<p<1$ is proved
and is used to obtain the uniqueness of the solutions to the  $L_p$
Minkowski problem for $0<p<1$.

For $p=0$, we introduce $(1-\tau)\diamond A_1\oplus_0\tau\diamond A_2$, the log-co-sum of two $C$-coconvex
sets $A_{1}$ and $A_{2}$ with respect to $\tau \in(0, 1)$, and prove the  log-Brunn-Minkowski inequality of $C$-coconvex sets. The
log-Minkowski inequality is also obtained and is applied to prove
the uniqueness of the solutions to the log-Minkowski problem that characterizes the cone-volume measures of $C$-coconvex sets.  Our
result solves an open problem raised by Schneider in [Schneider,
Adv. Math., 332 (2018), pp. 199-219].

 \end{abstract}

 \section{Introduction and overview of the main results}

Motivated by the elegant works by Khovanski\u{\i} and  Timorin
\cite{Kh2014} and Milman and Rotem \cite{MR2014}, Schneider in his
pioneer work \cite{Schneider2018} took the first  step to build up
the Brunn-Minkowski theory for $C$-coconvex sets.  Let $C$  be a
pointed closed convex cone in $\Rn$ with vertex at the origin $o$
and having nonempty interior.  The set $A\subset C$ is called a
$C$-coconvex set if the volume of $A$ is finite and
$A^{\bullet}=C\setminus A$ is  closed and convex (the set
$A^{\bullet}$ will be called a $C$-close set). The set $A^{\bullet}$
is said to be $C$-full when the $C$-coconvex set $A$ is bounded and
nonempty. For two $C$-coconvex sets $A_1$ and $A_2$ with
$A_1^{\bullet}=C\setminus A_1$ and $A_2^{\bullet}=C\setminus A_2$,
the ``co-sum" of $A_1$ and $A_2$, denoted by $A_1\oplus A_2$, is
defined as $A_1\oplus A_2=C\setminus (A_1^{\bullet}+A_2^{\bullet}),$
where ``+" denotes the usual Minkowski addition.  Note that
$A_1^{\bullet}+A_2^{\bullet}\subset C$ is  closed and convex.
Moreover, $V_n(A_1\oplus A_2)$,  the volume of $A_1\oplus A_2$,
indeed satisfies the following inequality
\begin{equation}\label{B-M-no-const} V_{n}\left(A_{1}\oplus
A_{2}\right)^{\frac{1}{n}}\leq
 V_{n}(A_{1})^{\frac{1}{n}}+
V_{n}(A_{2})^{\frac{1}{n}}.
\end{equation}
This shows that $V_n(A_1\oplus A_2)$ is finite and hence $A_1\oplus
A_2$ defines a $C$-coconvex set. Inequality \eqref{B-M-no-const} is
indeed equivalent to  \begin{equation}\label{brunn-minkowski ineq}
V_{n}\left((1-\lambda)A_{1}\oplus\lambda
A_{2}\right)^{\frac{1}{n}}\leq
(1-\lambda)V_{n}(A_{1})^{\frac{1}{n}}+\lambda
V_{n}(A_{2})^{\frac{1}{n}}
\end{equation}
for $\lambda\in (0, 1)$, where $V_n(\lambda A)=\lambda ^n V_n(A)$
for $\lambda A=\{\lambda x: x\in A\}$.   It has been also proved in
\cite{Schneider2018} that equality holds for \eqref{brunn-minkowski
ineq} if and only if $A_1=\alpha A_2$ for some $\alpha>0$.
Inequality \eqref{brunn-minkowski ineq} (or equivalently
\eqref{B-M-no-const}) is called the  Brunn-Minkowski inequality for
$C$-coconvex sets and plays essential roles in the development of
the  Brunn-Minkowski theory for $C$-coconvex sets in
\cite{Schneider2018}. In particular, Schneider used inequality \eqref{brunn-minkowski
ineq} to obtain the Minkowski inequality for $C$-coconvex sets:
\begin{equation} \label{Minkowski ineq}
\overline{V}_1(A_1, A_2)^n\le V_n(A_1)^{n-1}V_n(A_2),
\end{equation} with
equality if and only if $A_1=\alpha A_2$ for some $\alpha>0$ (see
Section \ref{prepar} for notations).

The Brunn-Minkowski inequality for $C$-coconvex sets reveals the
dissimilarity to the celebrated Brunn-Minkowski inequality for
convex bodies (convex compact sets in $\Rn$ with nonempty
interiors). The Brunn-Minkowski inequality for convex bodies reads:
for any $\lambda\in (0, 1)$ and for any two convex bodies $K_1$ and
$K_2$, one has  \begin{equation} V_{n}\left((1-\lambda)K_{1}+\lambda
K_{2}\right)^{\frac{1}{n}}\geq
(1-\lambda)V_{n}(K_{1})^{\frac{1}{n}}+\lambda
V_{n}(K_{2})^{\frac{1}{n}}. \label{B-M-convex bodies}
\end{equation}
Inequality  \eqref{B-M-convex bodies}  shows that $V_n(\cdot)$ is
$\frac{1}{n}$-concave in terms of the Minkowski sum of convex
bodies, while inequality \eqref{brunn-minkowski ineq} implies that  $V_n(\cdot)$ is $\frac{1}{n}$-convex in terms of the
co-sum of $C$-coconvex sets. In particular, inequalities
\eqref{brunn-minkowski ineq} and \eqref{B-M-convex bodies} have
opposite directions of inequalities; this has been carried over to
their equivalent Minkowski type inequalities.

On the other hand, the Brunn-Minkowski theories for $C$-coconvex
sets and for convex bodies also exhibit  similarity. For instance,
the surface area measures for $C$-coconvex sets can be defined in a
way rather similar to the surface area measures of convex bodies.
Let $C^{\circ}$ be the polar cone of $C$, i.e.,  $C^\circ=\{x\in
\mathbb{R}^{n}:x\cdot y\leq 0\ \  \mathrm{for\ all}\ \ y\in C\},$
where $x\cdot y$ means the usual inner product of  $x, y\in
\mathbb{R}^n$. Denote by $\partial E$ and $\mathrm{int} E$  the
boundary and the interior of $E\subset \Rn$, respectively.  Let
$S^{n-1}$ be the unit sphere in $\Rn$, $\Omega_C= S^{n-1}\cap
\mathrm{int} C^\circ$ and  $A^{\bullet}$ be a $C$-close set. Thus
the Gauss map $\nu_{A^{\bullet}}: \partial A^{\bullet} \cap
\mathrm{int} C\rightarrow S^{n-1}$ can be defined by the usual way:
$\nu_{A^{\bullet}}(x)$ is an outer unit  normal vector of $\partial
A^{\bullet}$ at  $x\in \partial A^{\bullet}\cap \mathrm{int} C$. The
surface area  measures  of $A^\bullet$ (see \cite[p.
201]{Schneider2018}) and $A$ on $\Omega_C$, denoted by
$S_{n-1}(A^{\bullet},\cdot)$ and $\overline{S}_{n-1}(A, \cdot)$
respectively, can then be defined by: for any Borel set $\eta\subset
\Omega_C$,
\begin{equation}\label{surface-area-close}\overline{S}_{n-1}(A,
\eta)=
S_{n-1}(A^{\bullet},\eta)=\mathcal{H}^{n-1}(\nu^{-1}_{A^{\bullet}}(\eta)),
\end{equation}
where $\nu^{-1}_{A^\bullet}$ denotes the reverse of
$\nu_{A^{\bullet}}$ and $\mathcal{H}^{n-1}$ is the
$(n-1)$-dimensional Hausdorff measure. Note that the surface area measure of $A^{\bullet}$ is uniquely determined, because the set of singular boundary points of $A^{\bullet}$ has $(n-1)$-dimensional Hausdorff measure zero. With the help of the surface
area measures in \eqref{surface-area-close}, $V_n(A_1)$ and
$\overline{V}_1(A_1, A_2)$ for $C$-coconvex sets $A_1$ and $A_2$
have the integral expressions (see \eqref{integral formula for
volume} and \eqref{mixed volume}, respectively) rather similar to
those for the volume and the mixed volume for convex bodies (see
e.g., \cite{Schneider2014}). Moreover, analogous to the classical
Minkowski problem \cite{min1897, Minkowski1903}, a Minkowski problem
aiming to characterize the surface area measures of $C$-coconvex
sets has been proposed by Schneider in \cite{Schneider2018}. A solution to this problem was obtained in \cite[Theorem 3]{Schneider2018} when the given finite Borel measure  has compact support in $\Omega_C$. This result was improved in \cite[Theorem 1]{Schneider2021} where the given Borel measure is only assumed to be finite (i.e., removing the restriction on its support). The uniqueness of the solution was also established in \cite[Theorem 2]{Schneider2018}. We would like to mention that the Minkowski problem to characterize the surface area measure for certain unbounded convex surfaces (or convex sets) has been done earlier by Chou and Wang \cite{CW-Unbounded}, Pogorelov \cite{Pogorelov}, and Urbas \cite{Urbas}.

The classical Minkowski problem has a long history which can be
traced back to \cite{Aleks1938, FJ1938, min1897, Minkowski1903}. In
his groundbreaking paper \cite{Lut1993}, Lutwak introduced the
elegant $L_p$ Minkowski problem extending the classical Minkowski
problem in a natural but nontrivial way. Since then, the $L_p$
Minkowski problem has attracted great attention, see e.g.,
\cite{Chen2006,CW2006,HLX2015,HLYZ2005,JLW2015,JLZ2016,LW2013,LYZ2004,
Umanskiy2003,GZhu2015I,GZhu2015II,GZhu2017} among others. In the
case $p=0$, the $L_p$ Minkowski problem for convex bodies reduces to
the log-Minkowski problem, and amazing progress has been made to this
challenging problem (see e.g., \cite{BHZ,BH2016,BLYZ2013,CLZ2019,HL2014,
stancu02,stancu08,zhug2014}).   No doubt that the $L_p$ Minkowski
problem has nice applications (see e.g., \cite{CLYZ2009,
HS2009b,LYZ2002,zhang1999}) and plays a central role in the
development of the $L_p$ Brunn-Minkowski theory of convex bodies, a
theory founded by Lutwak \cite{Lut1993, Lu96} based on the combination
of the volume and the $L_p$ addition of convex bodies. Contributions
to this $L_p$ theory include, e.g.,  \cite{CG2002a, Haberl2008, HS2009a, LYZ2000, LZ1997, MW2000, PW2012, SW2004, WY2008,
Ye2015a,YZZ2015, Zhao2016,ZZX2015}.

A major goal in this paper is to develop an $L_p$ Brunn-Minkowski
theory for $C$-coconvex sets, and provide the counterpart of the $L_p$ Brunn-Minkowski theory for convex bodies.
Our first contribution is to define an $L_p$ addition of
$C$-coconvex sets in Section \ref{p-sum}, which is called the
$p$-co-sum of $C$-coconvex sets. That is, for $0<p<1$ and two
$C$-coconvex sets $A_1, A_2$, the $p$-co-sum of $A_1$ and $A_2$,
denoted by $A_{1} \oplus_{p} A_{2}$, is  defined by $ A_{1}
\oplus_{p} A_{2}=C\setminus (A_{1} \oplus_{p} A_{2})^{\bullet}$ with
\begin{equation*}  (A_{1} \oplus_{p} A_{2})^\bullet =C\cap
\bigcap_{u\in \Omega_{C}} \Big\{x\in \mathbb{R}^n: x\cdot u\le
-\overline{h}\left(A_{1} \oplus_{p} A_{2}, u\right)
\Big\},\end{equation*}  where  $\overline{h}(A, \cdot):
\Omega_C\rightarrow \R$ is the support function of a
$C$-coconvex set $A$ on $\Omega_C$  (see (\ref{definition-suppport}) for the definition of $\overline{h}$), and for $0<p<1$, \begin{equation*}
\overline{h}\left(A_{1} \oplus_{p} A_{2},
u\right)=\big(\overline{h}\left(A_{1},
u\right)^{p}+\overline{h}\left(A_{2},
u\right)^{p}\big)^{\frac{1}{p}}, \ \ \text{for}\ \ u\in \Omega_C.
\end{equation*}
The set $A_{1} \oplus_{p} A_{2}$ turns out to be a $C$-coconvex set
and this result will be presented in Theorem
\ref{co-sum-set-1-def-theo}.  Moreover, we also establish  the $L_p$
Brunn-Minkowski inequality for $C$-coconvex sets in Theorem
\ref{p-B-M-I}: {\em Let $A_{1}, A_{2}$ be $C$-coconvex
sets. If $0<p<1$, then
\begin{equation*}
V_{n}(A_{1}\oplus_{p}A_{2})^{\frac{p}{n}}\leq
V_{n}(A_{1})^{\frac{p}{n}}+V_{n}(A_{2})^{\frac{p}{n}},
\end{equation*}
with equality if and only if $A_{1}=\alpha A_{2}$ for some
$\alpha>0$.} The above $L_p$ Brunn-Minkowski inequality for
$C$-coconvex sets reduces to inequality \eqref{B-M-no-const} if
$p=1$. It also shows that $V_n(\cdot)$ is $\frac{p}{n}$-convex in
terms of the $p$-co-sum of $C$-coconvex sets, which clearly
demonstrates its difference from the $L_p$ Brunn-Minkowski
inequality for convex bodies  (namely $V_n(\cdot)$ is
$\frac{p}{n}$-concave in terms of the $L_p$ addition of convex
bodies) \cite{Lut1993}. The above $L_p$ Brunn-Minkowski inequality
for $C$-coconvex sets provides  precisely the ``complementary"
analogue of the $L_p$ Brunn-Minkowski  inequality  for convex bodies
for $0<p<1$ conjectured by B\"{o}r\"{o}czky, Lutwak, Yang, and Zhang
in \cite{BLYZ2012}.  Note that this conjectured inequality is still
quite open in general and important progress has been made recently
by  Chen,  Huang, Li and  Liu \cite{CHLZ-2019}, and  Kolesnikov and
Milman \cite{K-M-2020}.

Our second contribution is to develop the $L_p$ surface area measures
of $C$-coconvex sets and study the related $L_p$ Minkowski problem for $0\neq p\in\mathbb{R}$.  Let
$\omega\subset \Omega_C$ be a compact set.  A  closed convex set
$ A^\bullet\subseteq C$ is  $C$-determined by $\omega$ if
$$
 A^\bullet=C\cap\bigcap_{u\in\omega}H^{-}(u,
h(A^\bullet,u)),
$$
where $h(A^\bullet, \cdot)$ denotes the support function of $A^\bullet$ on $\Omega_C$ (see its definition in (\ref{supp-c-c})).
The collection of all closed convex sets that are $C$-determined by
$\omega$ will be denoted by $\mathcal{K}(C,\omega).$  The $L_p$
surface area measure of a $C$-coconvex set $A$, with
$A^{\bullet}=C\setminus A\in \mathcal{K}(C,\omega)$, is denoted by
$\overline{S}_{n-1, p}\left(A,\cdot\right)$ and defined by
$$
\overline{S}_{n-1, p}\left(A,\cdot\right)=\overline{h}
\left(A,\cdot\right)^{1-p}\overline {S}_{n-1}\left(A,\cdot\right)\ \
\mathrm{on }\ \ \omega.
$$
The $L_p$ surface area measure of $A$ naturally appears in the
variational formula derived from the volume of the Wulff shape
associated with $(C, \omega, f_{\tau})$, where $f_{\tau}$ is given
by \eqref{def-f-tau}. This variational formula is presented and
proved in Theorem \ref{varitional-theorem}.  A fundamental question
related to the $L_p$ surface area measure is the following $L_p$
Minkowski problem (i.e., Problem \ref{lp-min-problem}): \vskip 2mm
\noindent {\bf The $L_p$ Minkowski problem.} {\em  Let  $0 \neq p\in
\R$ and $\omega\subset\Omega_C$ be a compact set. Under what
necessary and/or  sufficient conditions on a finite Borel measure
$\mu$ on $\omega$ does there exist a  set $A^\bullet\in
\mathcal{K}(C, \omega)$  with $A=C\setminus A^\bullet$ such that
$\mu= \overline {S}_{n-1, p}(A,\cdot)$?}

\vskip 2mm \noindent A solution to this $L_p$ Minkowski problem is
established in Theorem \ref{solution-lp-minkowski-1}, which reads:
{\em Let $\omega$  be a compact set of $\Omega_C$. Suppose that
$\mu$ is a nonzero finite Borel measure on $\Omega_C$ whose support
is concentrated on $\omega$.   For $0\neq p\in \R$, there exists a
$C$-full set $A_0^\bullet$ ($A_0=C\setminus A_0^\bullet$) such that
\begin{equation*}  \mu = c \cdot \overline{S}_{n-1,p}(A_0, \cdot) \ \ \ \mathrm{with} \ \  \ c=\frac{1}{nV_n(A_0)}
\left(\int_{\omega}\overline{h}(A_0, u)^{p}\,d\mu(u)\right).
\end{equation*}  }

Section \ref{M-inequality-5} aims to develop the $L_p$ Minkowski
inequality for the $L_p$ mixed volume of $C$-coconvex sets (see
Theorem \ref{9}).  In particular, we prove that for $0<p<1$ and  for
two $C$-coconvex sets $A_1, A_2$  such that $A_1^{\bullet}\in
\mathcal{K}(C, \omega)$ and $A_2^{\bullet}\in \mathcal{K}(C,
\omega)$, one has \begin{equation*}
\overline{V}_{p}(A_1,A_{2})=\dfrac{1}{n}\int_{\omega}\overline{h}(A_2,
u)^{p} \,d\overline{S}_{n-1, p}\left(A_1,u\right) \leq
V_{n}(A_1)^{\frac{n-p}{n}}V_{n}(A_{2})^{\frac{p}{n}},
\end{equation*}
with equality if and only if $A_1=\alpha A_{2}$ for some $\alpha>0$.
Again, this $L_p$ Minkowski inequality has its form similar to  the
$L_p$ Minkowski inequality for convex bodies \cite{Lut1993}; but
these two $L_p$ Minkowski inequalities have different directions of
inequalities and work for different ranges of $p$. One important
application of the $L_p$ Minkowski inequality for $C$-coconvex sets
is to obtain the uniqueness of the solutions to the $L_p$ Minkowski
problem, see Theorems \ref{14} and
\ref{solution-lp-minkowski-1-norm} for more details.  We would like
to emphasize that although the $L_p$ Minkowski problem for
$C$-coconvex sets resembles the $L_p$ Minkowski problem for convex
bodies in many ways (such as their formulations), these two problems
are completely different,  for instance,  the former one is solvable
for all $0\neq p\in \R$, but the existence of solutions to the
latter one is still unknown for many $p\in \R$ (in particular for
$p<-n$).

We also make contributions to develop the $L_0$ (or log)
Brunn-Minkowski theory for $C$-coconvex sets. In Section
\ref{section:p=0}, we introduce the log-co-sum of two $C$-coconvex
sets $A_{1}$ and $A_{2}$ with respect to $\tau \in(0, 1)$ (or, for simplicity, the log-co-sum of  $A_{1}$ and $A_{2}$), denoted by $(1-\tau)\diamond A_1\oplus_0\tau\diamond A_2$, whose support function takes
the following form:  \begin{equation*}
\overline{h}((1-\tau)\diamond  A_1\oplus_0\tau\diamond  A_2, u)=
\overline{h}\left(A_{1}, u\right)^{1-\tau}  \overline{h}\left(A_{2},
u\right)^{\tau}, \ \ \mathrm{for} \ u\in \Omega_C.
 \end{equation*}
See \eqref{definition-p-sum-6},
\eqref{relation-supp-close-6} and \eqref{def-concovex-log} for  more details. The
log-Brunn-Minkowski inequality of $C$-coconvex sets  is proved in
Theorem \ref{minus is convex-6}, which asserts that  $V_n(\cdot)$
is log-convex in terms of the log-co-sum; namely for $\tau \in(0,
1)$ and for two $C$-coconvex sets $A_{1}, A_{2}$,  \begin{equation*}
V_{n}((1-\tau)\diamond  A_{1}\oplus_{0}\tau\diamond  A_{2}) \leq
 V_n(A_{1})^{1-\tau} V_n(A_{2})^{\tau}.
\end{equation*}
Equality characterization for the log-Brunn-Minkowski inequality is
given in Theorem \ref{log is convex-6-1}, where we also establish
the  log-Minkowski inequality \eqref{log-ine-1-1-6}:  for two
nonempty $C$-coconvex sets $A_{1}$ and $A_{2}$,
\begin{equation*}
\overline{V}_0(A_1, A_2)=\dfrac{1}{n}\int_{\Omega_C}
\log\bigg(\frac{\overline{h}(A_2, u)}{\overline{h}\left(A_1,
u\right)}\bigg) \overline{h}\left(A_1,
u\right)\,d\overline{S}_{n-1}\left(A_1, u\right) \leq
\frac{V_n(A_1)}{n}\cdot
\log\bigg(\dfrac{V_{n}(A_{2})}{V_{n}(A_{1})}\bigg),
 \end{equation*}
with equality if and only if $A_1=\alpha A_2$ for some $\alpha>0$.
These inequalities provide  precisely ``complementary" analogues of
the log-Brunn-Minkowski and log-Minkowski inequalities for convex
bodies conjectured by B\"{o}r\"{o}czky, Lutwak, Yang, and Zhang in
\cite{BLYZ2012}. Note that the log-Brunn-Minkowski and log-Minkowski
inequalities for convex bodies are still quite open in general and
have received a lot of attention, see e.g., \cite{BK-2020, ML-2015,
Saro2015, Stancu16, YangZhang-2019}.  The significance of our
log-Minkowski inequality for $C$-coconvex sets (i.e., \eqref{log-ine-1-1-6})  can  also  be seen from the
fact that this inequality gives a positive answer to an open problem
raised by Schneider in \cite{Schneider2018}.  Indeed, Schneider in
\cite{Schneider2018} proposed the log-Minkowski problem for
$C$-coconvex sets aiming to  characterize  the cone-volume measures
of $C$-close sets, and also provided solutions to the log-Minkowski
problem in \cite[Theorems 4 and 5]{Schneider2018}. Schneider raised
an open problem regarding the uniqueness of the solutions to the
log-Minkowski problem for $C$-coconvex sets \cite[p.
203]{Schneider2018}. In Theorem \ref{14-log}, we use the the log-Minkowski inequality for
$C$-coconvex sets \eqref{log-ine-1-1-6} to confirm that solutions
to  the log-Minkowski problem for $C$-coconvex sets are indeed unique.

 \section{Background and preliminaries}\label{prepar}
The central objects of interest in this paper are  the closed convex
sets and related $C$-coconvex sets in the fixed pointed closed
convex cone $C$ defined in the Euclidean space $\Rn$. By $V_n(\cdot)$ we mean the volume. The
notations and definitions in this paper mainly follow those in
\cite{Schneider2014,Schneider2018} for consistence.

The dot product of two vectors $x, y\in \Rn$ is denoted by $x\cdot
y$.  A set $E\subset \Rn$ is said to be convex if  $\lambda
x+(1-\lambda)y\in E$  for any $\lambda\in [0, 1]$ and $x, y\in E$.
We say $C\subset \Rn$ a closed convex cone if $C$ is a closed and convex subset with nonempty interior such that  $\lambda x\in C$ for
any $x\in C$  and $\lambda \geq 0.$  Note that, if $C$ is a closed
convex cone, then $\lambda C = C$ for any $\lambda> 0$ and
especially $C + C = C$. A closed convex cone $C$ is called a pointed
cone if $-C\cap C=\{o\}$, where $o$ denotes the origin of $\Rn$ and
$-C=\{-x: x\in C\}$.

Denote by $\sphere=\big\{x\in \mathbb{R}^n: x\cdot x=1\big\}$  the unit sphere in $\Rn$. For $t\in \R$ and  $u\in \sphere$, let
$$H(u,t)=\big\{x\in \mathbb{R}^n: x\cdot u=t\big\}$$ and $H(u, t)$
is a hyperplane with normal vector $u\in \sphere$. Define the upper and lower halfspaces  $H^+(u,t)$ and $H^-(u,t)$, respectively, by the following formulas:
\begin{align*}
H^+(u,t)=\{x\in \mathbb{R}^n: x\cdot u\ge t\} \ \ \ \mathrm{and}\ \ \  H^-(u,t)=\{x\in \mathbb{R}^n: x\cdot u\le t\}.
\end{align*}

Let $C\subset \Rn$ be a fixed pointed closed convex cone. Associated
to $C$, there is a polar cone $C^\circ$ defined by $$C^\circ=\{x\in
\mathbb{R}^n: x\cdot y\le 0 \ \text{for all } y\in C\}. $$ Define
$\Omega_C$, an  open subset of $\sphere$, by  $ \Omega_C=
S^{n-1}\cap \text{int} C^\circ,$ where $\text{int} C^\circ$ is the
interior of $C^{\circ}$. It is easily checked that  the set
$H^+(u,t)\cap C$ is bounded with nonempty interior for any
$u\in\Omega_C$ and $t<0$. Moreover,  $H^+(u,0)\cap C=\{o\}$ for $u\in \Omega_C$.
Throughout the paper, let $\zeta\in S^{n-1}\setminus (S^{n-1}\cap
C^\circ)$ be fixed such that $x \cdot \zeta>0$ for all $x\in C\setminus
\{o\}$. The existence of such $\zeta$ is guaranteed by the fact that the
closed convex cone $C$ is  pointed. For simplicity, such a fixed
$\zeta\in \sphere$ will not appear in the following notations:
$$
H_t=\big\{x\in \mathbb{R}^n: x\cdot \zeta=  t\big\}, \ \   \
H^-_t=\big\{x\in \mathbb{R}^n: x\cdot \zeta\le t\big\},
$$
for $t\ge0$, $M_t= M\cap H^-_t$ for $M\subseteq C$ and for $t>0$.
Note that $M_t$ is always bounded for $t>0$. When $M=C$, we use
$C_t$ for the set $C\cap H^-_t$.

A set $L\subset\Rn$ is said to be a convex body if $L$ is a convex
compact set with nonempty interior.  Associated to a convex body $L$
is its support function $h(L, \cdot): \sphere\rightarrow \R$ which can be
formulated by $$h(L, u)=\max\big\{x\cdot u: x\in L\big\}.$$ We say that a sequence of
convex bodies $L_j, j\in \mathbb{N},$ converges to a convex body
$L_0$ in terms of the Hausdorff metric if $h(L_j, \cdot) \rightarrow
h(L_0, \cdot)$ uniformly on $\sphere$ as $j\rightarrow\infty$. The
support function can be used to uniquely determine a convex body.
Similarly, the support function of a $C$-close set $A^{\bullet}$  can be
defined in a similar manner. Recall that $A^{\bullet}\subset C$ is $C$-close
if $A^{\bullet}$ is a closed convex subset and the volume of $C \setminus A^{\bullet}$
is positive and finite.  A closed convex set $A^{\bullet}\subset C$ is said to
be $C$-full if $C\setminus  A^{\bullet}$ is bounded and nonempty.

If we let $h(A^{\bullet}, \cdot):  \text{int}C^\circ \rightarrow \R$ be the
support function of a $C$-close set $A^{\bullet}$, then $h(A^{\bullet}, \cdot)$ can be
formulated by
\begin{equation}
h(A^{\bullet}, x)=\sup \big\{x \cdot y: y \in A^{\bullet}\big\}, \ \ \text{for } \
x\in \text{int}C^\circ.\label{supp-c-c}
\end{equation}
Properties for the support functions of $C$-close sets are similar
to those for the support functions of convex bodies. For example,
$h(A^{\bullet}, \cdot)$ for a $C$-close set $A^{\bullet}$ is a sublinear function in
$\text{int}C^\circ$ and has positive homogeneity of degree $1$.
Again, the support function $h(A^{\bullet}, \cdot)$ and $A^{\bullet}$ can be (uniquely)
determined by each other, in particular, we can have
\begin{equation}\label{relation-supp-close} A^{\bullet}=C\cap \bigcap_{u\in
\Omega_{C}}H^{-}(u,h(A^{\bullet}, u)),
\end{equation}
where $H(u,h(A^{\bullet}, u))$ supports $A^{\bullet}$ at some point $y\in\partial A^{\bullet}$
with outer normal vector $u$. It is easily checked that
$-\infty<h(A^{\bullet}, \cdot)<0$ on $\Omega_C$ and the maximum can be indeed obtained in
\eqref{supp-c-c}.

The convergence of a sequence of $C$-close sets  is given as follows, see e.g.,  \cite{Schneider2018}. Let $ \mathbb{N}_0= \mathbb{N}\cup\{0\}$.

\begin{definition} \label{definition-convergence}
Let $\{A^{\bullet}_j\}_{j\in \mathbb{N}_0}$ be a sequence of $C$-close sets.
If there exists $t_0>0$ such that $A^{\bullet}_j\cap C_{t_0}\neq \emptyset$
for all $j\in\mathbb{N}$, and for all  $t\ge t_0$
$$
\lim_{j\rightarrow \infty}(A^{\bullet}_j\cap C_{t})= A^{\bullet}_0\cap C_{t},
$$ in terms of the Hausdorff metric,
then $\{A^{\bullet}_j\}_{j\in \mathbb{N}}$ is said to be convergent to $A^{\bullet}_0$, which is written by $A^{\bullet}_j\rightarrow A^{\bullet}_0$  as $j\rightarrow \infty. $
\end{definition}

Note that $o\notin A^\bullet$ if a $C$-coconvex set $A=C\setminus
A^\bullet$ is nonempty. A nice thing is that $\partial
A\setminus\partial C$ coincides with $\partial A^{\bullet}\setminus
\partial C$.  This nice property provides a good way to uniquely determine
$A$ through  $\overline{h}(A,\cdot): \text{int}C^\circ\rightarrow
\R$ formulated by
\begin{equation}\label{definition-suppport}
\overline{h}(A,x)=-h(A^\bullet,x), \ \ \text{for } \ x\in
\text{int}C^\circ.
\end{equation}
Hence, $\overline{h}(A, x+y)\geq \overline{h}(A, x)+\overline{h}(A,
y)$ and $\overline{h}(A, \lambda x)=\lambda \overline{h}(A, x)$ for
$\lambda>0$ and $x, y\in \text{int}C^\circ$. Clearly,
$0<\overline{h}(A, \cdot)<\infty$ in $\text{int}C^\circ$. Following \cite{Schneider2018},
we call $\overline{h}(A,\cdot)$ the support function of $A$.  In a
similar manner, one can also define $\overline{S}_{n-1}(A, \cdot)$,
the surface area measure of $A$, as follows: for any Borel set  $ \eta
\subset \Omega_{C}$,
\begin{equation*}
\overline{S}_{n-1}(A,\eta)=S_{n-1}\left(A^{\bullet},
\eta\right)=\mathcal{H}^{n-1}(\nu^{-1}_{A^\bullet}(\eta)),
\end{equation*}
where $S_{n-1}\left(A^{\bullet}, \cdot \right)$ is the surface area
measure of $A^{\bullet}$ defined in \eqref{surface-area-close}.  It
has been proved in \cite[Lemma 1]{Schneider2018} that
\begin{equation}\label{integral formula for volume}
V_n(A)=\frac{1}{n}\int_{\Omega_C} \overline{h}(A,u)
\,d\overline{S}_{n-1}(A,u).
\end{equation}
Again, $V_n(\lambda A)=\lambda^n V_n(A)$ for all $\lambda >0$ and
clearly \begin{equation}\label{monotonicity-volume} V_n(A_1)\le
V_n(A_2)
\end{equation}
for any two $C$-coconvex sets $A_1$ and $A_2$  such that $
A_1\subseteq A_2.$ The mixed volume of two $C$-coconvex sets $A_0,
A_1$ \cite[p. 219]{Schneider2018}, denoted by
$\overline{V}(A_0,\ldots,A_0, A_1)$ or by $\overline{V}_1(A_0, A_1)$
for short, is given by
\begin{equation}\label{mixed volume}
\overline{V}_1(A_0, A_1)=\overline{V}(A_0,\ldots,A_0, A_1)=\frac{1}{n}\int_{\Omega_C}
\overline{h}(A_1,u)\,d\overline{S}_{n-1}(A_0, u).
\end{equation}
Consequently, one has the following Minkowski inequality for
$\overline{V}_1(A_0, A_1)$ (see \cite[(27)]{Schneider2018}):
\begin{equation*}
\overline{V}_1(A_0, A_1)^n\le V_n(A_0)^{n-1}V_n(A_1),
\end{equation*} with
equality if and only if $A_0=\alpha A_1$ for some $\alpha>0$.
Moreover, (\ref{mixed volume}) implies $\overline{V}_1(A, A_1)\le \overline{V}_1(A, A_2)$ for $C$-coconvex sets  $A_1, A_2, A$  such that $A_1\subseteq A_2$.

\section{The $p$-co-sum of $C$-coconvex sets} \label{p-sum}
In this section, the $p$-co-sum of two $C$-coconvex sets for $0<p<1$
will be introduced and related properties will be provided. In
particular, we prove that the $p$-co-sum of $C$-coconvex sets for
$0<p<1$ is still a $C$-coconvex set. For $\alpha>0$ and  a
$C$-coconvex set $A$, let $\alpha \circ A=\alpha^{\frac{1}{p}} A$ if
there is no confusion on the value of $p$.

\begin{definition}\label{co-sum-set-1-def}
Let $A_{1}, A_{2}$ be two $C$-coconvex sets. For $p\in(0, 1)$ and
$\alpha_1,\alpha_2\geq 0$ (not both zero), define $\overline{h}:
\mathrm{int}C^\circ\rightarrow \R$ as follows: for all $x\in
\mathrm{int}C^\circ$, \begin{equation}\label{definition-p-sum}
 \overline{h}(x)=\big(\alpha_1 \overline{h}\left(A_{1},
x\right)^{p}+\alpha_2  \overline{h}\left(A_{2},
x\right)^{p}\big)^{\frac{1}{p}}.
 \end{equation}   \end{definition}

We would like to mention that Definition \ref{co-sum-set-1-def} also
works for $p=1$. In this case, it goes back to the co-sum of two
$C$-coconvex sets  (see  e.g., \cite{Schneider2018}):
 \begin{equation}\label{definition-co-sum}
 \overline{h}\left(\alpha_1 A_{1}\oplus\alpha_2 A_{2},
x\right)=\alpha_1\overline{h}\left(A_{1},
x\right)+\alpha_2\overline{h}\left(A_{2}, x\right),\ \ \text{for}\ \
x\in \text{int}C^\circ,
 \end{equation}
where the addition $``\oplus"$ is the co-sum given by $ \alpha_1
A_{1}\oplus\alpha_2 A_{2}= C\setminus (\alpha_1
A_{1}^\bullet+\alpha_2 A^\bullet_{2}).$  The case $p=0$ will be
discussed in Section \ref{section:p=0}.

The following lemma asserts that  $-\overline{h}$ is the support
function of a closed convex set contained in $C$.

 \begin{lemma} \label{minus is convex}
Let $p\in(0, 1)$ and $\alpha_1, \alpha_2 \geq 0$ (not both zero).
If $A_{1}$ and $A_{2}$ are nonempty $C$-coconvex sets, then
$-\overline{h}$, where $\overline{h}$ is given in
\eqref{definition-p-sum}, is a sublinear function in $\mathrm{int}
C^\circ$ and  hence uniquely determines a closed convex set
contained in $C$.
 \end{lemma}
\begin{proof}
First of all, as $A_1$ and $A_2$ are nonempty $C$-coconvex sets,
then $\overline{h}\left(A_{1}, x\right)>0$  and
$\overline{h}\left(A_{2}, x\right)>0$ for all $x\in
\mathrm{int}C^\circ$. As $\alpha_1 , \alpha_2\geq 0$ are not both
zero,  it is trivial to have $\overline{h}(x)>0$ for all $x\in
\mathrm{int}C^\circ$.

Secondly, for any $x\in \text{int}C^\circ$ and $t>0$,
\begin{eqnarray*}
\overline{h}\left( t x\right) =\big(\alpha_1\overline{h}\left(A_{1},
t x\right)^{p}+\alpha_2\overline{h}\left(A_{2}, t
x\right)^{p}\big)^{\frac{1}{p}}
=t\big(\alpha_1\overline{h}\left(A_{1},
x\right)^{p}+\alpha_2\overline{h}\left(A_{2},
x\right)^{p}\big)^{\frac{1}{p}}=t \overline{h}\left(x\right).
\end{eqnarray*}
That is, $\overline{h}$ and hence $-\overline{h}$ have
positive homogeneity of degree $1$ in $\text{int}C^\circ$.

Thirdly, let $T_{p}\left(a_{1},
a_{2}\right)=\left(a_{1}^{p}+a_{2}^{p}\right)^{\frac{1}{p}}$ for
positive numbers $a_1, a_2\in \mathbb{R}^{+}$ and $p\in(0,1)$. The
inverse Minkowski's inequality yields that, for any positive numbers
$a_{i}, b_{i}, c_{i}\in \mathbb{R}^{+} $  such that $a_{i}\geq
b_{i}+c_{i}$ for $i=1,2$,
\begin{equation} \label{triangle}
T_{p}\left(a_{1},a_{2}\right)\geq
T_{p}\left(b_{1},b_{2}\right)+T_{p}\left(c_{1},c_{2}\right).
\end{equation} Since $A_1, A_2$ are two $C$-coconvex sets,  for
$\lambda\in[0,1]$ and  $x,y\in \text{int}C^\circ$, one has
$$
h\left(C\setminus (\alpha^{1/ p}_i A_i), \lambda
x+(1-\lambda)y\right)\leq h\left(C\setminus (\alpha^{1/ p}_i A_i),
\lambda x\right)+ h\left(C\setminus (\alpha^{1/ p}_i A_i),
(1-\lambda) y\right),
$$
for $i=1,2$, or equivalently
$$
\overline{h}\left(\alpha^{1/ p}_i A_i, \lambda
x+(1-\lambda)y\right)\geq \overline{h}\left(\alpha^{1/ p}_i A_i,
\lambda x\right)+ \overline{h}\left(\alpha^{1/ p}_i A_i, (1-\lambda)
y\right).
$$
 For convenience, for $i=1,2$,  let
\begin{align*}
a_i=\overline{h}\big(\alpha^{1/ p}_i A_i, \lambda
x+(1-\lambda)y\big),  \ \
b_i =\overline{h}\big(\alpha^{1/ p}_i A_i, \lambda
x\big) \ \  \mathrm{and} \ \
c_i =\overline{h}\big(\alpha^{1/ p}_i A_i, (1-\lambda)y\big).
\end{align*}
 Employing (\ref{triangle}), one has
\begin{align*}
\overline{h}\left(\lambda x+(1-\lambda) y\right)
&=T_p(a_1, a_2)\ge T_{p}\left(b_{1},b_{2}\right)+T_{p}\left(c_{1},c_{2}\right) \\
&= \lambda\big(\alpha_1\overline{h}\left(A_{1},
x\right)^{p}+\alpha_2\overline{h}\left(A_{2},
x\right)^{p}\big)^{\frac{1}{p}}+(1-\lambda)\big(\alpha_1\overline{h}\left(A_{1},
y\right)^{p}+\alpha_2\overline{h}\left(A_{2},
y\right)^{p}\big)^{\frac{1}{p}}\\
&=\lambda \overline{h}(x)+ (1-\lambda) \overline{h}(y).
\end{align*}
Thus $ \overline{h}$ is concave and hence $-
\overline{h}$ is convex in $\text{int}C^\circ$.   Let
\begin{equation} \label{def-co-p-sum-2-2-2} (\copsum)^\bullet =C\cap
\bigcap_{u\in \Omega_{C}}H^{-}\Big(u,
-\overline{h}(u)\Big).\end{equation}  Clearly $(\copsum)^\bullet$ is
a closed convex set contained in $C$. Moreover, $o\notin
(\copsum)^\bullet $  as $-\overline{h}(x)<0$ for all $x\in
\mathrm{int}C^{\circ}$.    \end{proof}

Based on \eqref{def-co-p-sum-2-2-2}, we define
\begin{equation}\label{p-co-sum--1}
\copsum=C\setminus (\copsum)^\bullet.
\end{equation}
For convenience, we (formally) let
$\overline{h}(\copsum, \cdot)=\overline{h}$,  and
$\overline{h}$ becomes the support function of a $C$-coconvex set
once $\copsum$ is proved to be a $C$-coconvex set. To fulfill this
goal, one needs to prove that $V_n(\copsum)$ is finite and positive.
The following lemma is required.
\begin{lemma}\label{inclusion-1}
Let $A_{1}, A_{2}$ be $C$-coconvex sets, and $p\in (0, 1)$. For
$0<\lambda<1$, one has \begin{equation} \label{1.3} (1-\lambda)\circ
A_{1}\oplus_{p}\lambda\circ A_{2} \subseteq (1-\lambda) A_{1}\oplus
\lambda A_{2},
\end{equation}
with equality if and only if $A_{1}=A_{2}$.
\end{lemma}
\begin{proof} First of all, the function $t^{p}$ is strictly concave on $t\in(0,+\infty)$
if $0<p<1$. Thus,
\begin{equation*}
\big((1-\lambda)\overline{h}(A_{1},u)+\lambda
\overline{h}(A_{2},u)\big)^{p}\geq
(1-\lambda)\overline{h}(A_{1},u)^{p}+\lambda
\overline{h}(A_{2},u)^{p},
\end{equation*}
for all $u\in \Omega_C$. Equality holds if and only if
$\overline{h}(A_{1},u)=\overline{h}(A_{2},u)$ for all $u\in
\Omega_C$, namely, $A_1=A_2$.
 This, together with
(\ref{definition-p-sum}) and (\ref{definition-co-sum}), yields that
\begin{align*}
\overline{h}\left((1-\lambda) A_{1}\oplus\lambda  A_{2},u\right)&=
(1-\lambda)\overline{h}\left(A_{1},u\right)+\lambda
\overline{h}\left(A_{2},u\right)\\
&\geq \big((1-\lambda)\overline{h}\left(A_{1},u\right)^{p}+\lambda
\overline{h}\left(A_{2},u\right)^{p}\big)^{\frac{1}{p}}\\
&=\overline{h}((1-\lambda)\circ A_{1}\oplus_{p}\lambda\circ A_{2},
u).
\end{align*}
It follows from \eqref{def-co-p-sum-2-2-2} that
\begin{eqnarray*}   ((1-\lambda)\circ A_{1}\oplus_{p}\lambda\circ
A_{2})^\bullet & =& C\cap \bigcap_{u\in \Omega_{C}}H^{-}\Big(u,
-\overline{h}((1-\lambda)\circ A_{1}\oplus_{p}\lambda\circ A_{2},
u)\Big)\\  &\supseteq & C\cap \bigcap_{u\in \Omega_{C}}H^{-}\Big(u,
-\overline{h}\left((1-\lambda)  A_{1}\oplus\lambda
A_{2},u\right)\Big)\\ &=& \big((1-\lambda)  A_{1}\oplus\lambda
A_{2}\big)^{\bullet}.
\end{eqnarray*}
Thus, the desired argument \eqref{1.3} follows immediately from
\eqref{p-co-sum--1}:
\begin{eqnarray*}
(1-\lambda)\circ A_{1}\oplus_{p}\lambda\circ A_{2}
&=&C\setminus ((1-\lambda)\circ A_{1}\oplus_{p}\lambda\circ A_{2})^\bullet\\
&\subseteq & C\setminus \big((1-\lambda)  A_{1}\oplus\lambda
A_{2}\big)^{\bullet} \\ &=&  (1-\lambda)  A_{1}\oplus\lambda
A_{2}.\end{eqnarray*} Clearly, equality holds in \eqref{1.3} if and
only if $A_{1}=A_{2}$.
\end{proof}

 We now prove the  $L_p$ Brunn-Minkowski inequality
for $C$-coconvex sets for $0<p<1$. Note that the case $p=1$ (i.e., inequality \eqref{B-M-no-const}) has been discussed in \cite{Schneider2018}. The case $p=0$ will be discussed in Section \ref{section:p=0}.

\bt \label{p-B-M-I}
Let $A_{1}, A_{2}$ be $C$-coconvex sets. If $0<p<1$, then
\begin{equation}\label{1.5}
V_{n}(A_{1}\oplus_{p}A_{2})^{\frac{p}{n}}\leq
V_{n}(A_{1})^{\frac{p}{n}}+V_{n}(A_{2})^{\frac{p}{n}},
\end{equation}
with equality if and only if $A_{1}=\alpha A_{2}$ for some
$\alpha>0$.
\et

\begin{proof}
Let  $0<\lambda<1$. By (\ref{monotonicity-volume}), Lemma
\ref{inclusion-1},   and the Brunn-Minkowski inequality
(\ref{brunn-minkowski ineq}), we have
\begin{align*}
V_{n}(A_{1}\oplus_{p}A_{2})^{\frac{1}{n}}
&=V_{n}\left((1-\lambda)\circ( (1-\lambda)^{-\frac{1}{p}}
A_{1})\oplus_{p}\lambda\circ(\lambda^{-\frac{1}{p}} A_{2})\right)^{\frac{1}{n}}\\
&\leq V_{n}\left((1-\lambda)(1-\lambda)^{-\frac{1}{p}}A_{1}\oplus
\lambda\lambda^{-\frac{1}{p}}A_{2}\right)^{\frac{1}{n}}\\
&\leq(1-\lambda)V_{n}\left((1-\lambda)^{-\frac{1}{p}}A_{1}\right)^{\frac{1}{n}}+\lambda
V_{n}(\lambda^{-\frac{1}{p}}A_{2})^{\frac{1}{n}}\\
&=(1-\lambda)^{1-\frac{1}{p}}V_{n}(A_{1})^{\frac{1}{n}}
+\lambda^{1-\frac{1}{p}}V_{n}(A_{2})^{\frac{1}{n}}.
\end{align*}   Let
$\lambda=\dfrac{V_{n}(A_{2})^{\frac{p}{n}}}{V_{n}(A_{1})^{\frac{p}{n}}+V_{n}(A_{2})^{\frac{p}{n}}}$. Then
\begin{align*}
V_{n}(A_{1}\oplus_{p}A_{2})^{\frac{1}{n}}
\leq \dfrac{V_{n}(A_{1})^{\frac{p}{n}}+V_{n}(A_{2})^{\frac{p}{n}}}{\big(V_{n}(A_{1})^{\frac{p}{n}}
+V_{n}(A_{2})^{\frac{p}{n}}\big)^{1-\frac{1}{p}}}=\big(V_{n}(A_{1})^{\frac{p}{n}}+V_{n}(A_{2})^{\frac{p}{n}}\big)^{\frac{1}{p}}.
\end{align*} After rearrangement, one gets \eqref{1.5} as desired. The
equality condition, namely  $A_{1}=\alpha A_{2}$ for some
$\alpha>0$, follows from the equality conditions of
(\ref{brunn-minkowski ineq}) and (\ref{1.3}). In particular,
\begin{equation*}
V_{n}(A_{1}\oplus_{p}A_{2})^{\frac{p}{n}}=V_{n}((\alpha A_{2})\oplus_{p}A_{2})^{\frac{p}{n}}
=(1+\alpha^p) V_{n}(A_{2})^{\frac{p}{n}}=
V_{n}(A_{1})^{\frac{p}{n}}+V_{n}(A_{2})^{\frac{p}{n}}
\end{equation*} if $A_{1}=\alpha A_{2}$ for some $\alpha>0$.
\end{proof}

Indeed, it follows from Lemma \ref{minus is convex} and Theorem
\ref{p-B-M-I}  that  if $0<p<1$ and
$\alpha_1,\alpha_2>0$, then
\begin{eqnarray*}
V_n(\copsum)^{\frac{p}{n}} =
V_n\Big(\alpha_1^{\frac{1}{p}}A_1\oplus_p
\alpha_2^{\frac{1}{p}}A_2\Big)^{\frac{p}{n}} \leq
\alpha_1V_n(A_1)^{\frac{p}{n}} + \alpha_2V_n(A_2)^{\frac{p}{n}}
<\infty \end{eqnarray*} for any two $C$-coconvex sets $A_{1},
A_{2}$.  On the other hand, $\overline{h}(\copsum, x)>0$  for all
$x\in \mathrm{int}C^{\circ}$ as $\overline{h}(A_1, x)>0$ and
$\overline{h}(A_2, x)>0$, one has $V_n(\copsum)>0$. This observation
can be summarized as the following theorem.

 \bt \label{co-sum-set-1-def-theo} Let $A_{1}, A_{2}$ be two
$C$-coconvex sets. For $p\in(0, 1)$ and $\alpha_1,\alpha_2\geq 0$
(not both zero), the set $\copsum$ given in \eqref{p-co-sum--1} does
define a nonempty $C$-coconvex set, whose support function is $\overline{h}$
given by  \eqref{definition-p-sum}, namely, for all $x\in
\mathrm{int}C^\circ$,
\begin{equation*}\label{definition-p-sum-4-19}
 \overline{h}(\copsum, x)=\big(\alpha_1 \overline{h}\left(A_{1},
x\right)^{p}+\alpha_2  \overline{h}\left(A_{2},
x\right)^{p}\big)^{\frac{1}{p}}.
 \end{equation*}
 \et

\section{A variational formula of the volume related to the
$p$-co-sum}\label{variational}

In this section, we will prove a variational formula of the volume
of the Wulff shape associated with a family of  functions obtained from the
$p$-co-sum. Motivated by this variational formula, the $L_p$ surface area measure can be introduced and the related $L_p$
 Minkowski problem can be posed.  Recall that for a
$C$-close set $A^{\bullet}$, the surface area measure $S_{n-1}(A^{\bullet}, \cdot)$ defined on
$\Omega_C$  could be infinite. In order to better deal with the
surface area measure $S_{n-1}(A^{\bullet}, \cdot)$ and related Minkowski problems,
we will concentrate on a special class of $C$-close sets, namely
the sets that are $C$-determined by $\omega$.

Throughout the rest of the paper, we always let $\omega\subset
\Omega_C$ be a nonempty and compact set. As in the Brunn-Minkowski
theory for  convex bodies, Schneider in  \cite{Schneider2018}
introduced the $C$-coconvex analogue of the Wulff shape, which
provides a powerful tool in establishing the variational formula
regarding the co-sum  and plays essential roles in finding solutions
to the Minkowski problem that aims to characterize the surface area
measures of $C$-coconvex sets. For $f:\omega\rightarrow \mathbb{R}$
a positive and continuous function on $\omega$, define the Wulff
shape associated with $(C, \omega, f)$ to be a closed convex set of
the following form:
\begin{equation*}A^{\bullet}_f=C \cap \bigcap_{u \in \omega} H^{-}(u,-f(u)).
\end{equation*}
The Wulff shape $A^{\bullet}_f$ is a $C$-full set and $C\setminus  A^{\bullet}_f $ is
bounded and nonempty. Please see Figure \ref{fig-1-picture} for a
typical Wulff shape in $\R^2$.
\begin{figure}[htbp] \centering
\includegraphics[width=9cm]{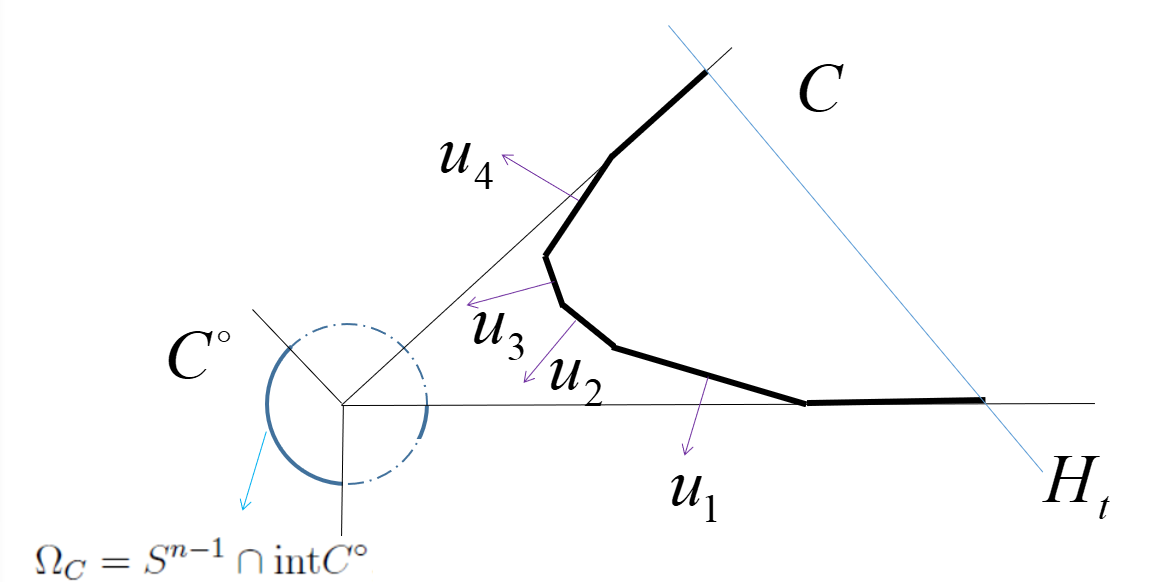}
\caption{A Wulff shape associated with $(C, \omega, f)$
for $\omega=\{u_1, u_2, u_3, u_4\}$ and $f: \omega \rightarrow (0, \infty)$ in
$\mathbb{R}^2$. The line $H_t$ is to illustrate Lemma \ref{lemma-bounded}.} \label{fig-1-picture}
\end{figure}

We also would like to mention that the Wulff shape $A^{\bullet}_f$ is
$C$-determined by $\omega$. Hereafter, a  closed convex set
$A^{\bullet}\subseteq C$ of the following form
$$
 A^{\bullet}=C\cap\bigcap_{u\in\omega}H^{-}(u,
h(A^{\bullet},u))
$$
is called  $C$-determined by $\omega$. The collection of  all closed
convex sets that are $C$-determined by $\omega$ is denoted by
$\mathcal{K}(C,\omega).$ Clearly,  $aA^{\bullet}\in \mathcal{K}(C,\omega)$
for any $a>0$ and for any $A^{\bullet}\in
\mathcal{K}(C,\omega)$. It was proved in \cite{Schneider2018} that if
$A^{\bullet}\in\mathcal{K}(C,\omega)$, then $A=C\setminus A^{\bullet}$
is bounded and $\overline{S}_{n-1}(A, \omega)=S_{n-1}(A^{\bullet},
\omega)$ is finite.

We summarize some important properties of the Wulff shape into the
following lemma for the easy future citation. Please read \cite[pp.
220-221]{Schneider2018} for more details.
\begin{lemma}\label{zero surface area}
Let $f: \omega\rightarrow \mathbb{R}$ be  a positive and continuous
function on $\omega$. Then \begin{equation}\label{compare functions}
h(A^{\bullet}_f,u)\le -f(u)\ \ \text{for}\ u\in\omega,
 \end{equation}    $h(A^{\bullet}_f, u) =-f(u)$ almost
everywhere with respect to the surface area measure $S_{n-1}(A^{\bullet}_f,
\cdot)$, $S_{n-1}(A^{\bullet}_f,\Omega_C\setminus\omega)=0$, and
 \begin{equation}\label{vol-func-1}  V_n(f):=
V_n(C\setminus A^{\bullet}_f)=\frac{1}{n}\int_\omega f(u) \,d S_{n-1}(A^{\bullet}_f, u).
 \end{equation}
\end{lemma}
Note that if $f=-h(A^{\bullet}, \cdot)$ for some $A^{\bullet}$ which is $C$-full and
$C$-determined by $\omega$, then  $A^{\bullet}_f=A^{\bullet}$ due to
\eqref{relation-supp-close}; in this case, one has
\begin{eqnarray*} \label{volume-support} V_n\big(-h(A^{\bullet}, \cdot)\big)= V_n(C\setminus
A^{\bullet})=\frac{1}{n}\int_\omega -h(A^{\bullet}, u) \,d S_{n-1}(A^{\bullet}, u).
 \end{eqnarray*}
It can be easily checked that $V_n(f_j)\rightarrow V_n(f)$  if
$f_j\rightarrow f$ uniformly on $\omega$. It has been proved in
\cite[Lemma 5]{Schneider2018} that the Wulff shape $A^{\bullet}_{f_j}$
associated with $(C, \omega, f_j)$  for $j\in \mathbb{N}$ converges
to the  Wulff shape $A^{\bullet}_{f_0}$ associated with $(C, \omega, f_0)$ if
the sequence of positive and continuous functions $f_j:
\omega\rightarrow \R$ converges uniformly to the positive and
continuous function $f_0: \omega\rightarrow \R$.  Moreover,
\cite[Lemma 6]{Schneider2018} asserts that if $\{A^{\bullet}_j\}_{j\ge1}$ is a
sequence in $\mathcal{K}(C,\omega)$ such that $A^{\bullet}_j\rightarrow A^{\bullet}_0$
for some $C$-full set $A^{\bullet}_0$, then $A^{\bullet}_0\in\mathcal{K}(C,\omega)$.

The following lemma is extracted from \cite[Lemma 8]{Schneider2018}.
See Figure \ref{fig-1-picture} for an illustration.
\begin{lemma} \label{lemma-bounded}
There is a constant $t>0$ with the following property: if  $A^{\bullet}\in
\mathcal{K}(C,\omega)$ and $V_n(C\setminus A^{\bullet})=1$, then $ C\cap H_t
\subset A^{\bullet}.$
\end{lemma}

The following lemma is the weak convergence of surface area measures
defined on $\omega$. Although its proof has already  been given in
the proof of  \cite[Lemma 7]{Schneider2018}, we still provide the
detailed proof here for completeness and for the convenience of
future citation.
\begin{lemma}\label{weak convergence}
Let $\{A^\bullet_i\}_{i\ge1}\subset\mathcal{K}(C,\omega)$ and
$A^\bullet_0\in\mathcal{K}(C,\omega)$. Let $A_i=C\setminus
A_i^{\bullet}$ for all $i\in \mathbb{N}_0$.   If
$A^\bullet_i\rightarrow A^\bullet_0$, then
$\overline{S}_{n-1}(A_i,\cdot)\rightarrow
\overline{S}_{n-1}(A_0,\cdot)$ weakly on $\omega$. That is, for any
continuous function $f: \omega\rightarrow \mathbb{R}$, one has,
\begin{equation}\label{weak-1}
\int_\omega f(u) \,d\overline{S}_{n-1}(A_i, u)\rightarrow \int_\omega
f(u) \,d\overline{S}_{n-1}(A_0,u).
\end{equation}
Moreover, if continuous functions $f_i: \omega\rightarrow
\mathbb{R}\ (i\in \mathbb{N})$ satisfy that $f_i\rightarrow f$
uniformly on $\omega$, then
\begin{equation}\label{weak-2}
\int_\omega f_i(u) \,d \overline{S}_{n-1}(A_i, u)\rightarrow \int_\omega
f(u) \,d \overline{S}_{n-1}(A_0,u).
\end{equation}
\end{lemma}

\begin{proof}
Suppose that  $A^\bullet_i\rightarrow A^\bullet_0$. It follows from
Definition \ref{definition-convergence} that for sufficiently large
$t>0$, one always has  $A^\bullet_i\cap C_t\rightarrow
A^\bullet_0\cap C_t$ in terms of the Hausdorff metric. Note that for
sufficiently large $t>0$, $A^\bullet_i\cap C_t$ for all $i\in
\mathbb{N}_0$ is a convex body. Hence, for sufficiently large $t>0$,
$$
S_{n-1}(A^\bullet_i\cap C_t,\cdot)\rightarrow
S_{n-1}(A^\bullet_0\cap C_t,\cdot) \ \ \text{weakly on}\ \ S^{n-1}.
$$ That is, for any continuous function $F: S^{n-1}\rightarrow
\mathbb{R}$ and  for sufficiently large $t>0$,
\begin{equation}\label{eq-1}
\int_{S^{n-1}} F(u) \,dS_{n-1}(A^\bullet_i\cap C_t,u)\rightarrow
\int_{S^{n-1}} F(u) \,dS_{n-1}(A^\bullet_0\cap C_t,u).
\end{equation}
Tietze's extension theorem implies that, for any continuous function $f:
\omega\rightarrow \mathbb{R}$, there is a continuous function $F:
S^{n-1}\rightarrow \mathbb{R}$ such that
$$F(u)=
\begin{cases}
f(u) & \text{on }  \omega,\\
0 & \text{on }   S^{n-1}\setminus\Omega_C.
\end{cases}
$$
Note that
$S^{n-1}=\omega\cup(\Omega_C\setminus\omega)\cup(S^{n-1}\setminus\Omega_C)$
and $S_{n-1}(A^\bullet_i, \Omega_C\setminus\omega)=0$ (see
\cite[(29)]{Schneider2018}). Consequently, for all $i\in
\mathbb{N}_0$,
\begin{eqnarray*}
\int_{S^{n-1}} F(u) \,dS_{n-1}(A^\bullet_i\cap C_t,u)
=\int_{\omega} f(u) \,dS_{n-1}(A^\bullet_i\cap C_t, u)
=\int_{\omega}f(u) \,dS_{n-1}(A^\bullet_i,u),
\end{eqnarray*}
where the equalities follow from $F=0$ on $S^{n-1}\setminus \Omega_C$ and the facts that, for sufficiently large $t>0$,
$S_{n-1}(A^\bullet_i\cap C_t,\cdot)=S_{n-1}(A^\bullet_i,\cdot)$ on
$\omega$ and $S_{n-1}(A_{i}^{\bullet}\cap C_{t},\Omega_C\setminus\omega)=0$. This, together with
(\ref{eq-1}) and the fact that
$\overline{S}_{n-1}(A_i,\cdot)=S_{n-1}(A^\bullet_i,\cdot)$, yields
that $$ \lim_{i\rightarrow\infty}\int_\omega f(u)
\,d\overline{S}_{n-1}(A_i, u)= \int_\omega f(u)\,d
\overline{S}_{n-1}(A_0,u).$$ This concludes the proof for the weak
convergence of  $\overline{S}_{n-1}(A_i,\cdot)\rightarrow
\overline{S}_{n-1}(A_0,\cdot)$ on $\omega$. In particular,
\begin{equation}\label{conv-whole-measure} \lim_{i\rightarrow
\infty}
\overline{S}_{n-1}(A_i,\omega)=\overline{S}_{n-1}(A_0,\omega).\end{equation}

Now let  $f_i\rightarrow f$ uniformly on  the compact set
$\omega\subset\Omega_C$ with $f$ and each $f_i$ being continuous for
$i\in \mathbb{N}$. For any $\varepsilon>0$,  there exists an $i_0\in
\mathbb{N}$, such that, for all $i>i_0$, one has $
|f_i(u)-f(u)|<\varepsilon$  for any $u\in\omega.$  Together with
(\ref{weak-1}), \eqref{conv-whole-measure},  and the fact that
$\overline{S}_{n-1}(A_0, \omega)=S_{n-1}(A_0^\bullet, \omega)$ is
finite as $A_0^{\bullet}\in \mathcal{K}(C, \omega)$, one gets,
\begin{align*}0& \leq  \lim_{i\rightarrow \infty}
\left|\int_\omega f_i(u) \,d\overline{S}_{n-1}(A_i,u)- \int_\omega
f(u) \,d\overline{S}_{n-1}(A_0,u)\right|\\
&\le  \lim_{i\rightarrow \infty} \left|\int_\omega f(u) \,d\overline{S}_{n-1}(A_i,u)-
\int_\omega
f(u) \,d\overline{S}_{n-1}(A_0,u)\right|+ \lim_{i\rightarrow \infty} \int_\omega \left| f_i(u) -
f(u)\right| \,d\overline{S}_{n-1}(A_i,u)\\
&\leq  \varepsilon
\overline{S}_{n-1}(A_0,\omega).
\end{align*}
The desired convergence (\ref{weak-2}) holds after taking $\varepsilon\rightarrow 0^+$.
\end{proof}

Let $A^{\bullet}\in \mathcal{K}(C, \omega)$ be $C$-determined by
$\omega$ and $f: \omega\rightarrow \mathbb{R}$ be  a  continuous
function on the compact set $\omega\subset \Omega_C$. Let
$A=C\setminus A^{\bullet}$. Define  $f_{\tau}: \omega\rightarrow
\mathbb{R}$ for $\tau\in (-\tau_0, \tau_0)$, where $\tau_0>0$ is a
constant small enough, by
\begin{equation}\label{def-f-tau}
f_{\tau}(u)= \Big(\overline{h}(A, u)^{p}+\tau f(u) \Big)^{1/p}\ \ \
\ \mathrm{for }\ u\in \omega.
\end{equation}
Note that both
$\overline{h}(A, \cdot)$ and $f$  are continuous functions on
the compact set $\omega$. Hence one can let $\tau_0$ be a constant
such that
$$
0<\tau_0<\frac{\min_{u\in \omega}\overline{h}(A, u)^p}{\max_{u\in
\omega}|f(u)|}.
$$
Clearly, for each $\tau\in (-\tau_0, \tau_0)$,   $f_{\tau}$
is also a  positive and continuous function on $\omega$.   It is
easy to verify that $f_{\tau} \rightarrow \overline{h}(A,
\cdot)$ uniformly on $\omega$ as $\tau\rightarrow 0$. Hence the
Wulff shape $A^{\bullet}_\tau$ associated with $(C, \omega, f_{\tau})$
converges to $A^{\bullet}$, by  \cite[Lemma 5]{Schneider2018}.
Moreover, it can be verified that
\begin{equation}\label{conv-unif-2-2-2}
\lim_{\tau\rightarrow 0}\frac{f_\tau(u)-\overline{h}(A,u)}{\tau}
=\dfrac{1}{p} f\left(u\right) \overline{h}\left(A,u\right)^{1-p} \ \
\mathrm{uniformly\ on }\ \ \omega.
\end{equation}

We are now ready to state and prove the variational formula
regarding $V_n(f_{\tau})=V_n(C\setminus A^{\bullet}_\tau)$. Although Theorem
\ref{varitional-theorem} can be obtained from \cite[Lemma
7]{Schneider2018} by letting $G(\tau, \cdot)$ in \cite[Lemma
7]{Schneider2018}  to be $f_{\tau}$, we again provide a detailed
proof in this paper for completeness.   Note that $f_0=
\overline{h}(A, \cdot)$ and hence $V_n(f_0)=V_n(A)$.

\bt \label{varitional-theorem} Let $A^{\bullet}\in \mathcal{K}(C,
\omega)$ be $C$-determined  by $\omega$ and $f: \omega\rightarrow
\mathbb{R}$ be  a  continuous function on  $\omega$. Let
$A=C\setminus A^{\bullet}$ and $f_{\tau}$ be defined by
\eqref{def-f-tau}. For all $0\neq p\in \R$, one has
\begin{equation}\label{1.6}
\frac{\,d V_n(f_{\tau})}{\,d\tau}\bigg|_{\tau=0}
=\lim_{\tau\rightarrow 0}\dfrac{V_{n}(f_{\tau})-V_{n}(f_0)}{\tau}
=\dfrac{1}{p}\int_{\omega}f(u) \ \overline{h}\left(A,u\right)^{1-p}\,d\overline
{S}_{n-1}\left(A,u\right).
\end{equation}
\et

\begin{proof}
Note that the Wulff shape $A^{\bullet}_\tau$ associated with $(C, \omega,
f_{\tau})$ is $C$-determined by $\omega$, and hence
$S_{n-1}(A^{\bullet}_\tau, \Omega_C\setminus\omega)=0$ (see \cite[p.
220]{Schneider2018}).  For convenience, let
$A_\tau=C\setminus A^{\bullet}_\tau.$  Formula \eqref{mixed
volume}, Lemma \ref{zero surface area} and the fact that
$\overline{S}_{n-1}(A_\tau,\cdot)=S_{n-1}(A^{\bullet}_\tau,\cdot)$
imply that   \begin{eqnarray} nV_{n}(f_{\tau})& = &
nV_{n}(A_\tau)
=\int_{\omega}f_\tau(u)\,d\overline{S}_{n-1}(A_\tau,u), \label{volume-1-1-1} \\
n\overline{V}_{1}(A_\tau, A)
& =& \int_{\omega}\overline{h}(A,u)\,d\overline{S}_{n-1}(A_\tau, u). \nonumber
\end{eqnarray}
Recall that $A^{\bullet}_\tau\rightarrow A^{\bullet}$, and hence
$V_n(A_\tau)\rightarrow V_n(A)$ by
\eqref{volume-1-1-1}. The uniform convergence of
\eqref{conv-unif-2-2-2}, together with Lemma  \ref{weak
convergence}, yields that
\begin{eqnarray}
\lim_{\tau\rightarrow
0^+}\frac{V_{n}(A_\tau)-\overline{V}_{1}(A_\tau, A)}{\tau}
&=& \frac{1}{n}\lim_{\tau\rightarrow
0^+}\int_{\omega}\dfrac{f_\tau(u)-\overline{h}(A,u)}
{\tau}\,d\overline{S}_{n-1}(A_\tau,u) \nonumber \\
&=& \frac{1}{np}\int_{\omega}f(u)\ \overline{h}
\left(A,u\right)^{1-p}\,d\overline{S}_{n-1}(A, u).\label{lim-inf-11}
\end{eqnarray}
Similarly, Lemma  \ref{zero surface area} (in particular,
\eqref{compare functions}) yields that
$\overline{h}(A_\tau, u)=-h(A^{\bullet}_\tau, u)\geq
f_{\tau}(u)$  for all $u\in \omega$, and hence
\begin{align*}
\liminf_{\tau\rightarrow
0^+}\frac{\overline{V}_{1}(A, A_\tau)-V_{n}(A)}{\tau}
&=\frac{1}{n}\liminf_{\tau\rightarrow
0^+}\int_{\omega}\dfrac{\overline{h}(A_\tau, u)-\overline{h}(A,u)}
{\tau}\,d\overline{S}_{n-1}(A,u)\\&\geq \frac{1}{n}\liminf_{\tau\rightarrow
0^+}\int_{\omega}\dfrac{f_\tau(u)-\overline{h}(A,u)}
{\tau}\,d\overline{S}_{n-1}(A,u)\\&=\frac{1}{np}\int_{\omega}f(u) \ \overline{h}
\left(A,u\right)^{1-p}\,d\overline{S}_{n-1}(A, u).
\end{align*} Together with the Minkowski inequality for
$C$-coconvex sets (\ref{Minkowski ineq}), one has, \begin{eqnarray}
 \frac{1}{np}\int_{\omega}f(u) \ \overline{h}
\left(A,u\right)^{1-p}\,d\overline{S}_{n-1}(A, u)
&\leq & \liminf_{\tau\rightarrow
0^{+}} \dfrac{\overline{V}_{1}(A, A_\tau)-V_{n}(A)}{\tau} \nonumber\\
 &\leq&  V_{n}(A)^{\frac{n-1}{n}}\liminf_{\tau\rightarrow
0^{+}}\dfrac{V_{n}(A_\tau)^{\frac{1}{n}}-V_{n}(A)^{\frac{1}{n}}}{\tau}. \label{lim-1}
\end{eqnarray}
Similarly, one can also get, by \eqref{lim-inf-11},
\begin{eqnarray}\label{lim-2}\nonumber
 \frac{1}{np}\int_{\omega}f(u)\ \overline{h}
\left(A,u\right)^{1-p}\,d\overline{S}_{n-1}(A, u)
&=& \limsup_{\tau\rightarrow
0^{+}}\dfrac{V_{n}(A_\tau)-\overline{V}_{1}(A_\tau, A)}{\tau}\\
\nonumber &\geq& \limsup_{\tau\rightarrow
0^{+}}\dfrac{V_{n}(A_\tau)-V_{n}(A_\tau)^{\frac{n-1}{n}}V_{n}(A)^{\frac{1}{n}}}{\tau}\\
&=&V_{n}(A)^{\frac{n-1}{n}}\limsup_{\tau\rightarrow
0^{+}} \dfrac{V_{n}(A_\tau)^{\frac{1}{n}}-V_{n}(A)^{\frac{1}{n}}}{\tau}.
\end{eqnarray}
By (\ref{lim-1}), (\ref{lim-2}) and the fact that
$\liminf\le\limsup$, we have
\begin{eqnarray*}
\frac{\,d \big(V_n(A_\tau)^{\frac{1}{n}}\big)}{\,d\tau}\bigg|_{\tau=0^+}
&=&\lim_{\tau\rightarrow
0^{+}}\frac{V_{n}(A_\tau)^{\frac{1}{n}}-V_{n}(A)^{\frac{1}{n}}}{\tau}\\
&=& \frac{V_{n}(A)^{\frac{1-n}{n}}}{np}\int_{\omega}f(u) \ \overline{h}
\left(A,u\right)^{1-p}\,d\overline{S}_{n-1}(A, u).
\end{eqnarray*}
Thus, by the L'Hospital rule and the fact that
$V_{n}(f_{\tau})=V_{n}(A_\tau)$, one gets
\begin{eqnarray*} \frac{\,d V_n(f_{\tau})}{\,d\tau}\bigg|_{\tau=0^+}
= \frac{\,d
\big(V_n(A_\tau)^{\frac{1}{n}}\big)}{\,d\tau}\bigg|_{\tau=0^+}
\cdot nV_{n}(A)^{\frac{n-1}{n}}= \frac{1}{p}\int_{\omega}f(u) \
\overline{h} \left(A,u\right)^{1-p}\,d\overline{S}_{n-1}(A, u).
\end{eqnarray*} Following the same lines, one can also get \begin{eqnarray*}
\frac{\,d V_n(f_{\tau})}{\,d\tau}\bigg|_{\tau=0^-} =
\frac{1}{p}\int_{\omega}f(u) \ \overline{h}
\left(A,u\right)^{1-p}\,d\overline{S}_{n-1}(A, u)
\end{eqnarray*} and hence the desired formula (\ref{1.6}) follows.
\end{proof}
Motivated by \eqref{1.6},  one can define the $L_p$ surface area
measure of $A^{\bullet}\in \mathcal{K}(C, \omega)$ (or equivalently
of $A=C\setminus A^{\bullet}$) and the $L_p$ mixed volume of $A$ and
a continuous function $g: \omega\rightarrow \R$ as follows.
\begin{definition} \label{Lp-measure-1}
Let $A^{\bullet}\in \mathcal{K}(C, \omega)$ be $C$-determined by
$\omega$ and $A=C\setminus A^{\bullet}$. For $0\neq p\in \R$, the
$L_p$ surface area measure of $A^{\bullet}$, denoted by $S_{n-1,
p}(A^{\bullet}, \cdot)$, on $\omega$ is absolutely continuous with
respect to $S_{n-1}(A^{\bullet}, \cdot)$ such that
$$
\frac{\,dS_{n-1, p}(A^{\bullet}, u)}{\,dS_{n-1}(A^{\bullet},
u)}=\big(\!-h(A^{\bullet}, u)\big)^{1-p}, \ \ \text{for} \ \
u\in\omega.
$$
Equivalently, the
$L_p$ surface area measure of $A$, denoted by $\overline{S}_{n-1,
p}(A, \cdot)$, is defined by $\overline{S}_{n-1, p}(A,
\cdot)=S_{n-1, p}(A^{\bullet}, \cdot)$, and hence
$$
\frac{\,d\overline{S}_{n-1, p}(A, u)}{\,d\overline{S}_{n-1}(A,
u)}=\overline{h}(A, u)^{1-p}, \ \ \text{for} \ \ u\in\omega.
$$

 Let $g: \omega\rightarrow \mathbb{R}$ be  a positive and continuous
function on  $\omega$. The $L_p$ mixed volume of $A$ and $g$, denoted by
$\overline{V}_p(A, g)$, is defined by
\begin{equation}\label{def-lp-mixed-volume-1} \overline{V}_p(A,
g)=\frac{1}{n}\int_{\omega}g(u)^{p}\ \overline{h}
\left(A,u\right)^{1-p}\,d\overline{S}_{n-1}(A,
u)=\frac{1}{n}\int_{\omega}g(u)^{p} \,d\overline{S}_{n-1, p}(A,
u).\end{equation}
 \end{definition}

A  fundamental question related to the $L_p$ surface area measure
is the following $L_p$ Minkowski problem. A solution to this $L_p$
Minkowski problem will be provided in Section \ref{section:
lp-minkowski-problem}.
\begin{problem}[The $L_p$ Minkowski problem] \label{lp-min-problem}
Let  $0 \neq p\in \R$ and $\omega\subset\Omega_C$ be a compact set.
Under what necessary and/or  sufficient conditions on a finite Borel
measure $\mu$ on $\omega$ does there exist a $C$-close set
$A^\bullet$ with $A=C\setminus A^\bullet$ such that $\mu= \overline
{S}_{n-1, p}(A,\cdot)$?
\end{problem}

\section{The $L_p$ Minkowski inequality and the unique
determination of $C$-coconvex sets for
 $0<p<1$}\label{M-inequality-5}

 In this section, the  $L_p$ Minkowski
inequality related to the $L_p$ mixed volume for $C$-coconvex sets for $0<p<1$ is
established. Such $L_p$ Minkowski inequality and the  $L_p$ Brunn-Minkowski inequality
(\ref{1.5}) can be viewed as the fundamental elements in the $L_p$
Brunn-Minkowski theory for $C$-coconvex sets.  Based on  the  $L_p$ Minkowski
inequality, the unique determination for $C$-coconvex sets by the $L_p$ surface area measure for $0<p<1$
is provided.

Indeed, based on \eqref{1.6} and \eqref{def-lp-mixed-volume-1}, if
$A=C\setminus A^{\bullet}$ with $A^{\bullet}\in \mathcal{K}(C,
\omega)$ and  $f: \omega\rightarrow \mathbb{R}$ is  a positive and
continuous function on a compact set $\omega$,  one has, for all
$0\neq p\in \R$,
\begin{equation}\label{mixed-lp-1-2-3} \frac{p}{n} \cdot
\lim_{\tau\rightarrow
0}\dfrac{V_{n}(f_{\tau})-V_{n}(f_0)}{\tau}=\dfrac{1}{n}\int_{\omega}f(u)\
\overline{h}\left(A,u\right)^{1-p}\,d\overline
{S}_{n-1}\left(A,u\right)=\overline{V}_p(A, f^{1/p}),
\end{equation}
where $f_{\tau}$ is defined by \eqref{def-f-tau}.
 \bd \label{p-mixed volume--1} Let  $A=C\setminus A^{\bullet}$ with
$A^{\bullet}\in \mathcal{K}(C, \omega)$ and   $A_1=C\setminus
A_1^{\bullet}$ with $A_1^{\bullet}\in \mathcal{K}(C, \omega)$.
Define the $L_p$ mixed volume of $A$ and $A_1$, denoted by
$\overline{V}_p(A, A_1)$, for $0\neq p\in \R$ as
\begin{equation} \label{mixed-lp-1-2-5}
\overline{V}_p(A, A_1)=\dfrac{1}{n}\int_{\omega}\overline{h}(A_1, u)^{p}\
\overline{h}\left(A,u\right)^{1-p}\,d\overline{S}_{n-1}\left(A,u\right)
=\dfrac{1}{n}\int_{\omega}\overline{h}(A_1, u)^{p} \,d\overline{S}_{n-1, p}\left(A,u\right).
\end{equation}
 \ed

From Definition \ref{p-mixed volume--1}, one can check that $\overline{V}_p(A, A)=V_n(A)$ and
$\overline{V}_{p}(\alpha A,\beta
A_1)=\alpha^{n-p}\beta^{p}\overline{V}_{p}(A,A_1)$ for any
$\alpha, \beta>0$ and two $C$-coconvex sets $A,A_1$ such that
$A^\bullet,A_1^\bullet\in\mathcal{K}(C,\omega)$.  We now prove the
following  $L_p$ Minkowski inequality.
 \bt \label{9} Let $A, A_1$ be
two $C$-coconvex sets such that $A^{\bullet}\in \mathcal{K}(C,
\omega)$ and $A_1^{\bullet}\in \mathcal{K}(C, \omega)$.  For
$0<p<1$, one has \begin{equation}\label{2.6}
\overline{V}_{p}(A,A_{1})^{n}\leq V_{n}(A)^{n-p}V_{n}(A_{1})^{p},
\end{equation}
with equality if and only if $A=\alpha A_{1}$ for some
$\alpha>0$.
\et

\begin{proof}
By \eqref{mixed-lp-1-2-5},  if $p=1$, one gets  (see also \eqref{mixed volume}),
\begin{equation*}
\overline{V}_1(A, A_1)=\frac{1}{n}\int_{\omega}
\overline{h}(A_1,u)\,d\overline{S}_{n-1}(A, u).
\end{equation*}
Moreover, $\overline{V}_1(A, A)=V_n(A)$ (see also \eqref{integral
formula for volume} and the fact that $S_{n-1}(A^\bullet, \Omega_C\setminus\omega)=0$
\cite[(29)]{Schneider2018}), and hence, a probability measure on $\omega$
can be defined as follows:
$$\,d\nu=\dfrac{\overline{h}(A,\cdot)}{nV_{n}(A)}\,d\overline{S}_{n-1}(A,\cdot).$$
It follows from the H\"{o}lder's inequality that, for $0<p<1$, one has
\begin{eqnarray}
\frac{\overline{V}_p(A, A_1)}{V_n(A)}&=&\dfrac{1}{nV_n(A)}\int_{\omega}\overline{h}(A_1, u)^{p}\
\overline{h}\left(A,u\right)^{1-p}\,d\overline{S}_{n-1}\left(A,u\right) \nonumber \\
&=&  \int_{\omega}\bigg(\frac{\overline{h}(A_1,u)}{\overline{h}(A,u)}\bigg)^{p}\,d\nu(u)\nonumber \\
\nonumber\\
&\leq &  \bigg(\int_{\omega}\frac{\overline{h}(A_1,u)}{\overline{h}(A,u)}\,d\nu(u)\bigg)^{p} \nonumber \\
&=&  \bigg(\int_{\omega}\frac{\overline{h}(A_1,u)}{nV_{n}(A)}\,d\overline{S}_{n-1}(A, u)\bigg)^{p}
\nonumber \\
&=& \bigg(\dfrac{\overline{V}_{1}(A,A_{1})}{V_{n}(A)}\bigg)^{p}.\label{2.1}
\end{eqnarray}
Employing the  Minkowski inequality \eqref{Minkowski ineq} to
\eqref{2.1}, one gets
\begin{eqnarray*}\overline{V}_p(A, A_1)^n \leq  V_n(A)^n\bigg(\dfrac{\overline{V}_{1}(A, A_{1})}
{V_{n}(A)}\bigg)^{np}\leq V_n(A)^n\bigg(\dfrac{V_n(A_{1})}{V_{n}(A)}\bigg)^{p}=
V_{n}(A)^{n-p}V_{n}(A_{1})^{p}.
\end{eqnarray*}
This is the desired inequality \eqref{2.6}. The characterization of
equality for inequality \eqref{2.6} is an easy consequence of the
characterization of equality for inequality \eqref{Minkowski ineq}
and for the H\"{o}lder's inequality (applied in \eqref{2.1}).
\end{proof}

\begin{remark}
Although inequality \eqref{2.6} was proved for special $C$-coconvex
sets, its proof indeed can be used to show that, for $0<p<1$ and  for any two
$C$-coconvex sets $A$ and $A_1$,
 \begin{equation*} \label{mixed-lp-1-2-5--7}
\overline{V}_p(A, A_1)=\dfrac{1}{n}\int_{\Omega_C}\overline{h}(A_1,
u)^{p}\
\overline{h}\left(A,u\right)^{1-p}\,d\overline{S}_{n-1}\left(A,u\right)\leq
V_{n}(A)^{\frac{n-p}{n}}V_{n}(A_{1})^{\frac{p}{n}},
\end{equation*}
with equality if and only if $A=\alpha A_{1}$ for some $\alpha>0$.
 \end{remark}

\bt \label{th5-2-6} Let $A_1,A_2$ be two $C$-coconvex sets such that
$A^\bullet_1, A_2^\bullet\in\mathcal{K}(C,\omega)$. For $0<p<1$, let $\overline{h}=(\overline{h}(A_1, \cdot)^p+\overline{h}(A_2, \cdot)^p)^{1/p}$ on $\omega$ and  $\bar{A}^{\bullet} \in \cK(C, \omega)$ be the Wulff shape associated with $(C, \omega,  \overline{h})$. Then
\begin{equation} V_{n}(\bar{A})^{\frac{p}{n}}\leq
V_{n}(A_{1})^{\frac{p}{n}}+V_{n}(A_{2})^{\frac{p}{n}}, \label{w-ineq-1}
\end{equation} where $\bar{A}=C\setminus \bar{A}^{\bullet}$,  is equivalent to the  $L_p$ Minkowski inequality \eqref{2.6}.
\et

\begin{proof} Let $\tau>0$ and $\bar{A}_{\tau}^{\bullet}\in \cK(C, \omega)$ be the Wulff shape associated with
$(C, \omega, \overline{h}_{\tau})$, where $\overline{h}_{\tau}=(\overline{h} (A_1, \cdot)^p+\tau \overline{h} (A_2, \cdot)^p)^{1/p}$ on $\omega$. Let $\bar{A}_{\tau}=C\setminus \bar{A}_{\tau}^{\bullet}$. Inequality  \eqref{w-ineq-1} for $0<p<1$
implies that, for all $\tau>0$,
\begin{equation*}
g(\tau)=V_n(\bar{A}_{\tau})^{\frac{p}{n}} -V_{n}(A_1)^{\frac{p}{n}}-\tau
V_{n}(A_{2})^{\frac{p}{n}} \leq  0.
\end{equation*} Taking use of Theorem
\ref{varitional-theorem} (or see \eqref{mixed-lp-1-2-3}), one gets
\begin{align*}
\lim_{\tau\rightarrow
0^{+}}\dfrac{g(\tau)-g(0)}{\tau}&=\lim_{\tau\rightarrow
0^{+}}\dfrac{V_{n}(\bar{A}_{\tau})^{\frac{p}{n}}-V_{n}(A_{1})^{\frac{p}{n}}-\tau
V_{n}(A_{2})^{\frac{p}{n}}}{\tau}\\
&=\dfrac{p}{n}V_{n}(A_{1})^{\frac{p}{n}-1}\lim_{\tau\rightarrow
0^{+}}\dfrac{V_{n}(\bar{A}_{\tau})-V_{n}(A_{1})}{\tau}-V_{n}(A_{2})^{\frac{p}{n}}\\
&=V_{n}(A_{1})^{\frac{p-n}{n}}\overline{V}_{p}\left(A_{1},A_{2}\right)
-V_{n}(A_{2})^{\frac{p}{n}}.
\end{align*} The desired  $L_p$ Minkowski inequality follows easily from $g(0)=0$ and  $g(\tau)\leq 0$ for $\tau>0$.

Conversely, we assume that the  $L_p$ Minkowski inequality holds. Note that  $\bar{A}^\bullet\in \mathcal{K}(C, \omega)$ and $\overline{h}=(\overline{h} (A_1, \cdot) ^p+ \overline{h} (A_2, \cdot)^p)^{1/p}$ on $\omega$.  For $A^\bullet\in \mathcal{K}(C,\omega)$,  by  \eqref{mixed-lp-1-2-5}, one
has,
\begin{align*}
\overline{V}_{p}\left(A, \bar{A}\right)&=\dfrac{1}{n}
\int_{\omega}\overline{h}\left( \bar{A}, u\right)^{p} \,d \overline{S}_{n-1, p}(A,u)\\
&=\dfrac{1}{n}\int_{\omega}\big(\overline{h}\left(A_1,u\right)^{p}
+\overline{h}\left(A_2,u\right)^{p}\big)\,d \overline{S}_{n-1, p}(A,u)\\
&=\dfrac{1}{n}\int_{\omega}\overline{h}\left(A_1,u\right)^{p}
\,d \overline{S}_{n-1, p}(A,u)+\dfrac{1}{n}\int_{\omega}\overline{h}
\left(A_2,u\right)^{p} \,d \overline{S}_{n-1, p}(A,u)\\
&=\overline{V}_{p}\left(A,A_1\right)+\overline{V}_{p}\left(A,A_2\right).
\end{align*}
Employing the  $L_p$ Minkowski inequality, one gets
\begin{eqnarray*}
V_{n}\left(\bar{A} \right) =
\overline{V}_{p}\left(\bar{A},  \bar{A}\right)= \overline{V}_{p}\left(\bar{A}, A_1\right)+\overline{V}_{p}
\left(\bar{A}, A_2\right) \leq V_{n}(\bar{A})^{\frac{n-p}{n}}\big(V_{n}(A_1)^{\frac{p}{n}}+V_{n}(A_2)^{\frac{p}{n}}\big).
\end{eqnarray*} Simplifying this, one can get the desired inequality \eqref{w-ineq-1}.
\end{proof}

\begin{remark} Let $A_1,A_2$ be two $C$-coconvex sets such that
$A^\bullet_1, A_2^\bullet\in\mathcal{K}(C,\omega)$. For $0<p<1$, one cannot expect to have $\bar{A}=A_1\oplus _p A_2.$ Indeed, it is easily checked that $\bar{A}\subseteq A_1\oplus _p A_2$ because $\bar{A}^{\bullet}=C\setminus \bar{A}$ is the $C$-coconvex set generated by less halfspaces. However, counterexamples show that $\bar{A}\neq A_1\oplus _p A_2$ can happen, and an example is provided in Figure \ref{fig-2-picture}. Moreover, as $\bar{A}\subseteq A_1\oplus _p A_2$, it follows from Theorem \ref{p-B-M-I} that
\begin{equation*} V_n(\bar{A})\leq
V_{n}(A_{1}\oplus_{p}A_{2})^{\frac{p}{n}}\leq
V_{n}(A_{1})^{\frac{p}{n}}+V_{n}(A_{2})^{\frac{p}{n}}.
\end{equation*} Clearly, if equality holds in \eqref{w-ineq-1}, then equality holds in \eqref{1.5} as well, and this requires that  $A_{1}=\alpha A_{2}$ for some $\alpha>0$.  On the other hand, if $A_{1}=\alpha A_{2}$ for some $\alpha>0$, then  $$\bar{A}=A_1\oplus _p A_2=(\alpha A_{2})\oplus_{p}A_{2}=(1+\alpha^p)^{1/p}A_2,$$  which clearly implies the equality  in \eqref{w-ineq-1}.  This implies that the equality characterizations for inequalities \eqref{2.6} and \eqref{w-ineq-1} are indeed the same.    \begin{figure}[htbp] \centering
\includegraphics[width=8.2cm]{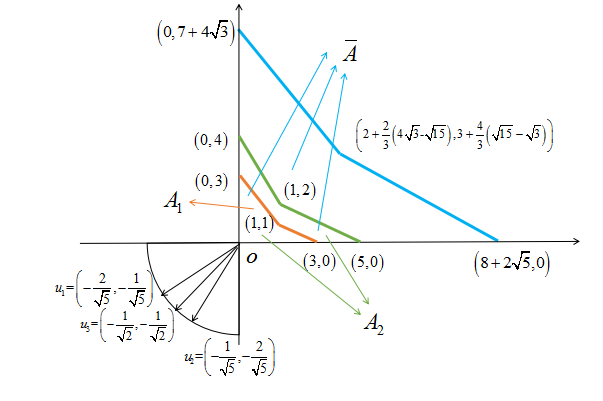}
\caption{Let $\omega=\{u_1, u_2\}$ and $p=1/2$.  The support functions of  $\bar{A}$ and $A_1\oplus _p A_2$ at $u_3$ are $\frac{15+4\sqrt{3}+2\sqrt{15}}{3\sqrt{2}}$ (about $6.9943$) and $\frac{5\sqrt{2}+4\sqrt{3}}{2}$ (about $6.9996$),  respectively.}    \label{fig-2-picture}
\end{figure}
 \end{remark}

The following result is for the unique determination of $C$-coconvex
sets by the $L_p$ surface area measure  for $0<p<1$. In particular, it can be applied to obtain the
uniqueness of solutions to the $L_p$ Minkowski problem (i.e.,
Problem \ref{lp-min-problem}), if its solution exists; see Theorem
\ref{solution-lp-minkowski-1-norm} for more details.

\bt \label{14} Let $A_1=C\setminus A_1^{\bullet}$ and
$A_2=C\setminus A_2^{\bullet}$ be two $C$-coconvex sets such that
$A^\bullet_1, A_2^\bullet\in\mathcal{K}(C,\omega)$. Then
$A_{1}=A_{2}$, if the following identity for $0<p<1$  holds on
$\omega$: \begin{equation}\label{assumption-1}
\overline{h}(A_{1},\cdot)^{1-p}\overline{S}_{n-1}(A_{1},\cdot)
=\overline{h}(A_{2},\cdot)^{1-p}\overline{S}_{n-1}(A_{2},\cdot).\end{equation}
\et
\begin{proof} By \eqref{mixed-lp-1-2-5} and \eqref{assumption-1}, one gets
\begin{eqnarray*}
\overline{V}_p(A_1, A_2)
&=&\dfrac{1}{n}\int_{\omega}\overline{h}(A_2, u)^{p}\ \overline{h}
\left(A_1,u\right)^{1-p}\,d\overline{S}_{n-1}\left(A_1,u\right)\\
&=& \dfrac{1}{n}\int_{\omega}\overline{h}(A_2, u)^{p}\
\overline{h}\left(A_2,u\right)^{1-p}\,d\overline{S}_{n-1}\left(A_2,u\right)=V_n(A_2).
\end{eqnarray*}
The $L_p$ Minkowski inequality  \eqref{2.6} yields
\begin{equation}
\dfrac{V_{n}(A_{2})}{V_{n}(A_{1})}=\dfrac{\overline{V}_{p}(A_{1},A_{2})}{V_{n}(A_{1})}\leq
 \dfrac{V_{n}(A_{1})^{\frac{n-p}{n}}V_{n}(A_{2})^{\frac{p}{n}}}{V_{n}(A_{1})}\leq
\left(\dfrac{V_{n}(A_{2})}{V_{n}(A_{1})}\right)^{\frac{p}{n}}.\label{inequality-com-1}
\end{equation}
This implies $V_{n}(A_{2})\leq V_{n}(A_{1})$. Similarly, we have
$V_{n}(A_{2})\geq V_{n}(A_{1})$.
Thus $V_{n}(A_2)=V_{n}(A_1)$ and
inequality \eqref{inequality-com-1} becomes equality. This further
leads to $A_{1}=A_{2}$.
\end{proof}

More results for the unique determination of $C$-coconvex sets can
be obtained. For example, we have the following theorem. \bt
\label{11} Let $A_1,A_2$ be two $C$-coconvex sets such that their
companion sets $A^\bullet_1,A_2^\bullet\in\mathcal{K}(C,\omega)$.
Let $V_{n}(A_1)\geq V_{n}(A_2)$ and $0<p<1$. \vskip 2mm  \noindent
(i) If $V_{n}(A_1)\leq \overline{V}_{p}(A_1,A_2)$, then $A_1=A_2$.
\vskip 2mm  \noindent (ii) If $V_{n}(A_1)\leq
\overline{V}_{p}(A_2,A_1)$, then $A_1=A_2$. \vskip 2mm  \noindent
(iii) If $\overline{V}_{p}(A,A_1)=\overline{V}_{p}(A,A_2)$ for any
$C$-coconvex set $A$ such that $A^\bullet\in\mathcal{K}(C,\omega)$,
then $A_1=A_2$. \et

\begin{proof} Only the proof of the
first case will be demonstrated and other cases can be proved along
the same lines. By the $L_p$ Minkowski inequality (\ref{2.6}), under
the assumption in Theorem \ref{11}, one has
\begin{equation*}
V_{n}(A_1)^{n}\leq \overline{V}_{p}(A_1,A_2)^{n}\leq
V_{n}(A_1)^{n-p}V_{n}(A_2)^{p},
\end{equation*}
with equality in the second inequality if and only if $A_1=\alpha
A_2$ for some $\alpha>0$. The above inequality yields that
$V_{n}(A_1)\le V_{n}(A_2)$. Thus $V_{n}(A_1)=V_{n}(A_2)$ by the
condition, and then $A_1=A_2$.
\end{proof}

\section{A solution to the $L_p$ Minkowski problem} \label{section: lp-minkowski-problem}
In this section, we will provide a solution to the $L_p$ Minkowski
problem (i.e., Problem  \ref{lp-min-problem}) for all $0\neq
p\in\mathbb{R}$. Solutions to the $L_p$ Minkowski problem for $p=1$
and $p=0$  have been provided in \cite{Schneider2018}. \bt
\label{solution-lp-minkowski-1} Let $\omega$  be a compact set of
$\Omega_C$. Suppose that $\mu$ is a nonzero finite Borel measure on
$\Omega_C$ whose support is concentrated on $\omega$. For $0\neq
p\in \R$, there exists a $C$-full set $A_0^\bullet$ ($A_0=C\setminus
A_0^\bullet$) such that
\begin{equation}\label{non-norm-p-min} \mu = c \cdot
\overline{S}_{n-1,p}(A_0, \cdot)  \ \ \ \mathrm{with}\ \ \ c=\frac{1}{nV_n(A_0)}
 \int_{\omega}\overline{h}(A_0, u)^{p}\,d\mu(u) . \end{equation}
 \et
\begin{proof}
Denote by $C^+(\omega)$  the set of all continuous and
positive functions on $\omega$. For any $\alpha>0$ and $f\in
C^+(\omega)$, one can check that the Wulff shape $A^{\bullet}_{\alpha f}$
associated with $(C,\omega, \alpha f)$ is a dilation of the Wulff
shape $A^{\bullet}_{f}$ associated with $ (C,\omega,f)$, namely $A^{\bullet}_{\alpha f}
=\alpha A^{\bullet}_f.$  Recall \eqref{vol-func-1} that $V_n(f)=V_n(C\setminus
A^{\bullet}_f)$ and hence \begin{equation}\label{hom-v-f} V_n(\alpha
f)=\alpha^n V_n(f).
\end{equation}

For $p>0$, consider the following optimization problem:
\begin{equation}  \label{opt-min-1} \sup \Big\{\mathcal{L}(f):   \ \
f\in C^+(\omega) \Big\},
\end{equation}
where $\mathcal{L} (f)$ is the functional on $C^{+}(\omega)$ defined
by \begin{equation} \mathcal{L} (f)=V_n(f)^{-\frac{p}{n}}
\int_{\omega}f(u)^{p}\,d\mu(u). \label{opt-functional-1-p}
\end{equation}  It follows from \eqref{hom-v-f} that the
optimization problem \eqref{opt-min-1}  has homogeneity of degree
$0$.

The key to solve the $L_p$ Minkowski problem is to find a $C$-full
set $A_0^{\bullet} \in \mathcal{K}(C, \omega)$, such that
$\overline{h}({A_0}, \cdot)$ with $A_0=C\setminus
A_0^\bullet$ is positive on $\omega$ and is a solution to the
optimization problem \eqref{opt-min-1}.  Indeed, if such
$A_0^{\bullet}$ exists and $\overline{h}(A_0, \cdot)$ solves
\eqref{opt-min-1}, then for any continuous function $g:
\omega\rightarrow \R$, one must have
$$\frac{\partial \mathcal{L} \big((\overline{h}(A_0, \cdot)^{p}+t
g)^{\frac{1}{p}}\big)}{\partial t} \bigg|_{t=0}=0,$$ due to the fact
that the optimization problem \eqref{opt-min-1}  has homogeneity of
degree $0$. This, together with Theorem \ref{varitional-theorem},
further yields that, for any continuous function $g:
\omega\rightarrow \R$,   \begin{equation} \frac{1}{n}
\bigg(\int_{\omega}\overline{h}(A_0, u)^{p}\,d\mu(u)\bigg) \cdot
\bigg( \int_{\omega}g(u) \overline{h}(A_0, u)^{1-p}
\,d\overline{S}_{n-1}(A_0, u) \bigg)=V_n(A_0)\cdot
\int_{\omega}g(u)\,d\mu(u).\label{Lag-meth-1}
\end{equation} As $g$ is arbitrary, one can get
\begin{equation*}
\mu = \frac{1}{nV_n(A_0)}
\bigg(\int_{\omega}\overline{h}(A_0, u)^{p}\,d\mu(u)\bigg) \cdot
\overline{h}(A_0, \cdot)^{1-p} \overline{S}_{n-1}(A_0, \cdot)= c \cdot
\overline{S}_{n-1,p}(A_0, \cdot).
\end{equation*}
Then $A_0$ satisfies \eqref{non-norm-p-min} as desired.

Now let us claim that the optimization problem \eqref{opt-min-1} has
$\overline{h}(A_0, \cdot)$ as one of its optimizers for some $A_0=C\setminus
A_0^\bullet$ such that $A_0^{\bullet}\in \mathcal{K}(C, \omega)$ is a
$C$-full set. To this end, by Lemma \ref{zero surface area}, one
gets for any $f\in C^+(\omega)$, $V_n(f)= V_n(C\setminus
A^{\bullet}_ f)=V_n\big(\overline{h}(C\setminus  A^{\bullet}_ f, \cdot)\big)$  and
$$
\overline{h}(C\setminus  A^{\bullet}_ f, u)
=-h(A^{\bullet}_ f, u)\geq f(u), \ \ \text{for} \ \ u\in\omega.
$$
It follows from \eqref{opt-functional-1-p} that
$\mathcal{L}(\overline{h}\big(C\setminus  A^{\bullet}_ f, \cdot)\big)\geq \mathcal{L}(f)$ for
$p>0$. Hence an optimizer for
\begin{equation}  \label{opt-min-geo} \sup
\Big\{ \mathcal{L}\big(\overline{h}(Q, \cdot)\big):   \ \ Q=C\setminus Q^\bullet \ \
\mathrm{such\ that} \ \ Q^\bullet \in \mathcal{K}(C, \omega) \Big\},
\end{equation}
for $p>0$ is also an optimizer of the optimization problem
\eqref{opt-min-1}.  Again the optimization problem
\eqref{opt-min-geo} has homogeneity of degree $0$, and hence it is
equivalent to the following optimization problem:  for $p>0$,
\begin{equation}  \label{opt-min-geo-p} \Theta=\sup
\Bigg\{\int_{\omega}\overline{h}(Q, u)^{p}\,d\mu(u):    \
Q=C\setminus Q^\bullet \ \mathrm{such\ that}\  V_n(Q)=1 \
\mathrm{and} \ Q^\bullet \in \mathcal{K}(C, \omega) \Bigg\}.
\end{equation}
Clearly, the supremum is taken over a nonempty set:
\begin{equation*}
\mathcal{Q}=\Big\{ Q=C\setminus Q^\bullet: \ \   V_n(Q)=1 \
\mathrm{and} \ Q^\bullet \in \mathcal{K}(C, \omega) \Big\}\neq
\emptyset.
\end{equation*}

Recall that $\zeta\in \sphere$ is a special and fixed unit vector which
will not appear in the following notations (see Section
\ref{prepar}): for $t\ge0$,  $$ H_t=\big\{x\in \mathbb{R}^n: x\cdot
\zeta=  t\big\} \ \ \mathrm{and}\ \   \ H^-_t=\big\{x\in \mathbb{R}^n:
x\cdot \zeta \le t\big\}.
$$
Lemma \ref{lemma-bounded} implies that there exists a constant
$t_0>0$, such that  $C\cap H_{t_0}\subset Q^{\bullet}$ for any $Q\in
\mathcal{Q}$. Hence,
\begin{equation*}
\int_{\omega}\overline{h}(Q, u)^{p}\,d\mu(u) \leq \int_{\omega}(-h(C\cap
H_{t_0}, u))^{p}\, d\mu (u):=c_0.
\end{equation*}
Note that $c_0$ is a universal constant and is  independent of the
choice of $Q\in \mathcal{Q}$. This shows that the optimization
problem \eqref{opt-min-geo-p} is well-defined and the supremum
$\Theta$ is finite. Moreover $\Theta>0$.

Take a sequence $\{Q_{i}\}_{i\in \mathbb{N}}\subset \mathcal{Q}$
such that
\begin{equation*} 0< \Theta= \lim_{i\rightarrow\infty}
\int_{\omega}\overline{h}(Q_i, u)^p\,d\mu(u).
\end{equation*}
Again, by Lemma \ref{lemma-bounded},  the constant $t_0>0$ as above
satisfies that  $C\cap H_{t_0}\subset Q_i^{\bullet}$ for any $i\in
\mathbb{N}$. In particular, $\emptyset\neq Q_{i}^{\bullet}\cap
H_{t_0}^{-}\subset C\cap H_{t_0}^{-}$ for all $i\in \mathbb{N}$.
Therefore,  the sequence of  convex bodies $\{Q_{i}^{\bullet}\cap
H_{t_0}^{-}\}_{i\in \mathbb{N}}$ is bounded and the Blaschke
selection theorem can be applied to get  a subsequence, say
$\{Q_{i_k}^{\bullet}\cap H_{t_0}^{-}\}_{k\in \mathbb{N}}$,
converging to some compact convex set, say $K$. Due to the fact that $C\cap H_{t_0}\subset Q_i^{\bullet}$ for any $i\in
\mathbb{N}$, there exists a closed convex set $A_0^{\bullet}\subset C$ such that $K=A_0^{\bullet}\cap H_{t_0}^{-}$ and  $Q_{i_k}^{\bullet}\rightarrow A_0^{\bullet}$.  Moreover, $V_n(A_0)=V_n(C\setminus A_0^{\bullet})=\lim_{k\rightarrow \infty} V_n(C\setminus Q_{i_k}^{\bullet})=1$. This further gives that $A^\bullet_0\in \cK(C, \omega)$ is a $C$-full set. It is also easily checked that $\overline{h}(Q_{i_k}, \cdot)$ converges
to $\overline{h}(A_0, \cdot)$ uniformly on $\omega$. We can also claim that
$\overline{h}(A_0, \cdot)$ is positive on $\omega$ by contradiction. That
is, if  $\overline{h}(A_0, u_0)=0$ for some $u_0\in \omega$, then
$$0=-\overline{h}(A_0, u_0)= h(A_0^\bullet, u_0)=\max\{x\cdot u_0:
x\in A_0^{\bullet}\}. $$ Note that $\omega\subset \Omega_C$ is a
compact set, and $x\cdot u_0<0$ for all $o\neq x\in C$. This further
yields $o\in A_0^{\bullet}$ and hence
$\overline{h}(A_0, u)=h(A_0^{\bullet}, u)=0$ for all $u\in \omega$.
In this case,
$$A_0^{\bullet}=C\cap \bigcap_{u\in\omega}H^{-}(u,
h(A_0^\bullet,u))=C\cap \bigcap_{u\in\omega}H^{-}(u,
0)=C$$ and then $A_0=C\setminus A_0^\bullet=\emptyset$. This
contradicts $V_n(A_0)=1$ and hence  $\overline{h}(A_0, \cdot)$ is
positive on $\omega$.  In conclusion,   $A_0\in \mathcal{Q}$ and
\begin{equation*}
\Theta= \lim_{k\rightarrow\infty}
\int_{\omega}\overline{h}(Q_{i_k}, u)^p\,d\mu(u)=
\int_{\omega}\overline{h}(A_0, u)^p \,d\mu(u).
\end{equation*}
That is, $A_0$ solves the optimization problem \eqref{opt-min-geo-p}
(and equivalently \eqref{opt-min-geo} and \eqref{opt-min-1}).  This
completes the proof for the case $p>0$.

Now let $p<0$. In this case, we consider the following optimization problem:
\begin{equation}  \label{opt-min-p<0}
\inf \Big\{\mathcal{L}(f):   \ \ f\in C^+(\omega) \Big\}.
\end{equation}
Again the optimization problem \eqref{opt-min-p<0} has homogeneity
of degree $0$. Moreover, by Lemma \ref{zero surface area} and
\eqref{opt-functional-1-p}, one gets
$\mathcal{L}\big(\overline{h}(C\setminus K_f, \cdot)\big)\leq \mathcal{L}(f).$
Hence an optimizer for
\begin{equation}  \label{opt-min-geo-p<0}
 \inf \Bigg\{\int_{\omega}\overline{h}(Q, u)^{p}\,d\mu(u):
 \ Q=C\setminus Q^\bullet \ \mathrm{such\ that}\  V_n(Q)=1
 \  \mathrm{and} \ Q^\bullet \in \mathcal{K}(C, \omega) \Bigg\}
\end{equation}
for $p<0$ is also an optimizer of the optimization problem
\eqref{opt-min-p<0}.  Clearly the optimization problem
\eqref{opt-min-geo-p<0} is well defined.     Repeating the arguments
in the case $p>0$, one can find $A_0^\bullet \in \mathcal{K}(C,
\omega)$ which is also $C$-full, such that $A_0\in \mathcal{Q}$
solves the optimization problem \eqref{opt-min-geo-p<0}.
Consequently, \eqref{Lag-meth-1} holds for any continuous function
$g: \omega\rightarrow \R$ and the desired formula
\eqref{non-norm-p-min} follows.
  \end{proof}

   \bt \label{solution-lp-minkowski-1-norm}
Let $\omega$  be a compact set of $\Omega_C$. Suppose that $\mu$ is
a nonzero finite Borel measure on $\Omega_C$ whose support is
concentrated on $\omega$.  If  $p\in \R$ and $p\neq 0, n$, then
there exists a $C$-full set $A^\bullet$ ($A=C\setminus A^\bullet$)
such that   \begin{equation}\label{sol-uni-p} \mu=\overline{S}_{n-1,
p} (A,
\cdot)=\overline{h}(A,\cdot)^{1-p}\overline{S}_{n-1}(A,\cdot).
\end{equation}
Moreover, if $0<p<1$, the solution to the $L_p$ Minkowski problem is unique.
\et
 \begin{proof}
By Theorem \ref{solution-lp-minkowski-1}, there exists $A_0$ such
that $A_0^{\bullet}=C\setminus A_0\in \mathcal{K}(C, \omega)$ is a
$C$-full set and \eqref{non-norm-p-min} holds, namely $\mu =c \cdot
  \overline{S}_{n-1, p}(A_0, \cdot).$
The easily checked fact that $  \overline{S}_{n-1, p}(A_0, \cdot)$ has homogeneity of degree $n-p$ implies
that $A=c^{\frac{1}{n-p}} A_0$ satisfies \eqref{sol-uni-p}, as desired. The uniqueness of the solutions  to the $L_p$ Minkowski problem for
$0<p<1$ follows immediately from Theorem \ref{14}.
\end{proof}

\section{The  log-Brunn-Minkowski and log-Minkowski inequalities} \label{section:p=0}

In this section, the log-co-sum of two $C$-coconvex sets will be
introduced and related properties will be provided. In particular,
we prove that the log-co-sum of $C$-coconvex sets is still a
$C$-coconvex set. The  log-Brunn-Minkowski and log-Minkowski
inequalities will be established. The log-Minkowski
inequality is applied to confirm that
the solutions to the log-Minkowski problem (i.e., Problem
\ref{lp-min-problem} for $p=0$), which aims to characterize the
cone-volume measures of $C$-coconvex sets, are unique. Hence an open problem raised by Schneider in \cite[p.
203]{Schneider2018} is solved.

Let us begin with the limit of (\ref{definition-p-sum}) as $p\rightarrow 0^{+}$. Let $A_{1}$ and $A_{2}$ be two nonempty $C$-coconvex sets. For $\tau \in(0, 1)$ and $0<p<1$, then for all $x\in
\mathrm{int}C^\circ$,
\begin{equation}\label{1}
\lim_{p\rightarrow 0^{+}}\left[(1-\tau)\overline{h}\left(A_{1}, x\right)^{p}+\tau\overline{h}\left(A_{2}, x\right)^{p}\right]^{\frac{1}{p}}=\overline{h}\left(A_{1}, x\right)^{1-\tau}  \overline{h}\left(A_{2},
x\right)^{\tau}.
 \end{equation}
For simplification, we write
\begin{equation}\label{definition-p-sum-6}
\overline{h}_\tau(x)=
\overline{h}\left(A_{1}, x\right)^{1-\tau}  \overline{h}\left(A_{2},
x\right)^{\tau}.
 \end{equation}

\begin{lemma}\label{concavity-h-tau-11}
Let $\tau\in(0, 1)$, and $A_{1}, A_{2}$ be nonempty
$C$-coconvex sets. Then the function $\overline{h}_\tau: \mathrm{int}C^\circ\rightarrow (0, \infty)$ is concave and has positive homogeneity of degree $1$ in $\text{int}C^\circ$.
\end{lemma}
\begin{proof} As $A_1$ and $A_2$ are nonempty $C$-coconvex sets,
one sees that both $\overline{h}\left(A_{1}, x\right)$ and
$\overline{h}\left(A_{2}, x\right)$ are positive for all $x\in
\mathrm{int}C^\circ$. Thus, $\overline{h}_\tau(x)>0$ for all $x\in
\mathrm{int}C^\circ$. It is easy to verify that $
\overline{h}_\tau(sx) =s\cdot
\overline{h}_\tau(x)$ for any
$x\in \text{int}C^\circ$ and $s>0$. This shows that $\overline{h}_\tau(\cdot)$ has positive homogeneity of degree $1$ in $\text{int}C^\circ$. Let $\lambda\in (0, 1)$ and $x,y\in \text{int}C^\circ$. It follows from \eqref{1}, \eqref{definition-p-sum-6}, and  (the proof of) Lemma \ref{minus is convex} (in particular, (\ref{triangle})) that \begin{align*}
\overline{h}_\tau(\lambda
x+(1-\lambda)y)
&= \lim_{p\rightarrow 0^{+}}\left[(1-\tau) \overline{h}\left(A_{1},
\lambda x+(1-\lambda)y\right)^{p} + \tau \overline{h}\left(A_{2},
\lambda x+(1-\lambda)y\right)^{p}\right]^{\frac{1}{p}}  \\
&\geq \lim_{p\rightarrow 0^{+}}\lambda\left[(1-\tau)\overline{h}(A_{1}, x)^{p}+\tau\overline{h}(A_{2}, x)^{p}\right]^\frac{1}{p}\\
&\quad\ +\lim_{p\rightarrow 0^{+}}(1-\lambda)\left[(1-\tau)\overline{h}(A_{1}, y)^{p}+
\tau\overline{h}(A_{2}, y)^{p}\right]^\frac{1}{p}\\
&=\lambda \overline{h}_\tau(x)+(1-\lambda)
\overline{h}_\tau(y).
\end{align*} This shows that $\overline{h}_\tau$ is concave in $\text{int}C^\circ$ as desired. \end{proof}

By Lemma \ref{concavity-h-tau-11},   $-
\overline{h}_\tau$  is convex, strictly negative, and has positive homogeneity of degree $1$ in $\text{int}C^\circ$.   Therefore, $-
\overline{h}_\tau$ uniquely determines a closed convex set contained in $C$. Indeed, one can restrict $-
\overline{h}_\tau$ to  $\Omega_C=S^{n-1}\cap \mathrm{int}C^{\circ}$  and let
\begin{align}\label{relation-supp-close-6}
((1-\tau)\diamond A_1\oplus_0\tau \diamond  A_2)^{\bullet}&=C\cap \bigcap_{u\in
\Omega_{C}}H^{-}(u,-\overline{h}_\tau(u)).
\end{align} Consequently, for all $u\in \Omega_C$
\begin{equation*}
h(((1-\tau)\diamond A_1\oplus_0\tau \diamond A_2)^{\bullet}, u)=-\overline{h}_\tau (u).
 \end{equation*} Moreover, $o\notin  \big((1-\tau)\diamond A_1\oplus_0\tau \diamond A_2\big)^{\bullet} $
as $\overline{h}_\tau>0$ on $\Omega_C$.
 Based on \eqref{relation-supp-close-6}, let
\begin{equation}\label{def-concovex-log}
(1-\tau)\diamond A_1\oplus_0\tau\diamond A_2=C\setminus
\big((1-\tau)\diamond A_1\oplus_0\tau\diamond  A_2\big)^{\bullet}
\end{equation} and this set will be called the log-co-sum of $C$-coconvex sets $A_1$ and $A_2$ with respect to $\tau$.  Note that,  for any  $u\in\Omega_C$,  \begin{equation}\label{support-def-log}
\overline{h}((1-\tau)\diamond A_1\oplus_0\tau\diamond A_2,u)=\overline{h}\left(A_{1}, u\right)^{1-\tau}  \overline{h}\left(A_{2},
u\right)^{\tau}.
\end{equation}

The following theorem states the log-Brunn-Minkowski inequality for two $C$-coconvex sets, and hence shows that the set $(1-\tau)\diamond A_1\oplus_0\tau \diamond   A_2$ is a nonempty $C$-coconvex set.

 \bt \label{minus is convex-6}
Let $\tau\in(0, 1)$. If $A_{1}$ and $A_{2}$ are nonempty
$C$-coconvex sets, then
\begin{equation}\label{1.5-6}
V_{n}((1-\tau)\diamond  A_{1}\oplus_{0}\tau \diamond  A_{2}) \leq
 V_n(A_{1})^{1-\tau} V_n(A_{2})^{\tau}.
\end{equation}
 \et
\begin{proof}
By the arithmetic-geometric inequality and \eqref{definition-p-sum-6}, it follows that for all $p\in (0, 1)$ and $u\in \Omega_C$,
\begin{align*}
\overline{h}((1-\tau)\circ A_{1}\oplus_{p}\tau\circ A_{2},
u)&=\big[(1-\tau)\overline{h}\left(A_{1},u\right)^{p}+\tau
\overline{h}\left(A_{2},u\right)^{p}\big]^{\frac{1}{p}}\\
&\geq \overline{h}\left(A_{1}, u\right)^{1-\tau}  \overline{h}\left(A_{2},
u\right)^{\tau}=\overline{h}_\tau(u).
\end{align*}
 Thus, by \eqref{def-co-p-sum-2-2-2}, \eqref{relation-supp-close-6} and Theorem \ref{co-sum-set-1-def-theo}, one gets
\begin{eqnarray*}
((1-\tau)\diamond  A_1\oplus_0\tau\diamond  A_2)^\bullet & =& C\cap \bigcap_{u\in
\Omega_{C}}H^{-}\Big(u, -\overline{h}_\tau(u)\Big)\\  &\supseteq & C\cap \bigcap_{u\in \Omega_{C}}H^{-}\Big(u,
-\overline{h}\left((1-\tau)\circ  A_{1}\oplus_p \tau\circ
A_{2},u\right)\Big)\\ &=& \big((1-\tau) \circ  A_{1}\oplus_p \tau
\circ A_{2}\big)^{\bullet}.
\end{eqnarray*}
Taking the complement with respect to $C$,  one gets
  \begin{eqnarray*}
  (1-\tau)\diamond  A_1\oplus_0\tau\diamond  A_2\subseteq   (1-\tau) \circ A_{1}\oplus_p\tau\circ  A_{2}.\label{necce-1}
  \end{eqnarray*}
This further yields, by the  $L_p$ Brunn-Minkowski inequality
(\ref{1.5}), that
\begin{eqnarray}V_n((1-\tau)\diamond  A_1\oplus_0\tau\diamond  A_2)&\leq& V_n((1-\tau)
\circ A_{1}\oplus_p\tau \circ A_{2}) \nonumber\\&\leq& \big((1-\tau)
V_n(A_{1})^{\frac{p}{n}}+\tau
V_n(A_{2})^{\frac{p}{n}}\big)^{\frac{n}{p}}<\infty. \label{comp-o-p}
\end{eqnarray}
Therefore, $(1-\tau)\diamond A_1\oplus_0\tau \diamond A_2$ has finite volume and hence is a
$C$-coconvex set. To obtain inequality
\eqref{1.5-6}, one lets $p\rightarrow 0^+$ in
\eqref{comp-o-p}, and then
\begin{eqnarray*}
V_n((1-\tau)\diamond  A_1\oplus_0\tau\diamond  A_2)\leq \lim_{p\rightarrow 0^+}
\big((1-\tau) V_n(A_{1})^{\frac{p}{n}}+\tau
V_n(A_{2})^{\frac{p}{n}}\big)^{\frac{n}{p}}= V_n(A_{1})^{1-\tau}
V_n(A_{2})^{\tau}.
\end{eqnarray*}
This completes the proof of the theorem.   \end{proof}

Unfortunately, the proof in Theorem \ref{minus is convex-6} does not
give the characterization of equality for  \eqref{1.5-6}. However, by the fact that the set $(1-\tau)\diamond A_1\oplus_0\tau \diamond   A_2$ is a nonempty $C$-coconvex set proved in Theorem \ref{minus is convex-6}, the
following theorem establishes the  log-Minkowski inequality and the
log-Brunn-Minkowski inequality \eqref{1.5-6} with characterization
of equalities. Define the $L_0$ (or log) mixed volume of two
nonempty $C$-coconvex sets $A_1$ and $A_2$ by
\begin{eqnarray}\label{def-log-mix} \overline{V}_0(A_1,
A_2)=\dfrac{1}{n}\int_{\Omega_C} \log\bigg(\frac{\overline{h}(A_2,
u)}{\overline{h}\left(A_1, u\right)}\bigg) \overline{h}\left(A_1,
u\right)\,d\overline{S}_{n-1}\left(A_1, u\right),\end{eqnarray} provided the integral exists and is finite.

 \bt \label{log is convex-6-1}
Let $A_{1}$ and $A_{2}$ be two nonempty $C$-coconvex sets. The
following  log-Minkowski inequality holds with equality if and only
if $A_1=\alpha A_2$ for some $\alpha>0$:
\begin{equation} \label{log-ine-1-1-6}
\frac{ \overline{V}_0(A_1, A_2)}{V_n(A_1)}\leq \frac{1}{n}
\cdot \log\bigg(\dfrac{V_{n}(A_{2})}{V_{n}(A_{1})}\bigg).
 \end{equation}
Moreover, the  log-Brunn-Minkowski inequality \eqref{1.5-6} holds
with equality if and only if  $A_1=\alpha A_2$ for some $\alpha>0$.
 \et

\begin{proof}
The proof follows along the lines for those in the proofs of Theorems
\ref{9} and \ref {th5-2-6}.  Indeed,
 a probability measure on $\Omega_C$ can be defined as follows:
$$
\,d\nu=\dfrac{\overline{h}(A_1,\cdot)}{nV_{n}(A_1)}\,d\overline{S}_{n-1}(A_1,\cdot).
$$
It follows from the Minkowski inequality \eqref{Minkowski ineq}  and Jensen's
inequality (as the logarithmic function is strictly concave) that
\begin{eqnarray}
\frac{\overline{V}_0(A_1, A_2)}{V_n(A_1)}
&=&\int_{\Omega_C}\log\bigg(\frac{\overline{h}(A_2,u)}{\overline{h}(A_1,u)}\bigg)
\,d\nu(u)\nonumber \\
&\leq & \log \bigg(\int_{\Omega_C}\frac{\overline{h}(A_2,u)}{\overline{h}(A_1,u)}\,d\nu(u)\bigg) \nonumber \\
&=&\log  \bigg(\int_{\Omega_C}\frac{\overline{h}(A_2,u)}{nV_{n}(A_1)}\,d\overline{S}_{n-1}(A_1, u)\bigg)
\nonumber \\
&=&\log \bigg(\dfrac{\overline{V}_{1}(A_1,A_{2})}{V_{n}(A_1)}\bigg)\nonumber \\&\leq&\frac{1}{n}\cdot
\log\bigg(\dfrac{V_{n}(A_{2})}{V_{n}(A_{1})}\bigg).\label{2.1-log}
\end{eqnarray}
This is the desired inequality \eqref{log-ine-1-1-6}. The
characterization of equality for inequality \eqref{log-ine-1-1-6} is
an easy consequence of the characterization of equality for
inequality \eqref{Minkowski ineq} and for Jensen's inequality
(applied in \eqref{2.1-log}).

Although the log-Brunn-Minkowski inequality \eqref{1.5-6} has
already been proved in Theorem \ref{minus is convex-6}, it can also
be obtained by taking use of  inequality \eqref{log-ine-1-1-6} as
follows: as $(1-\tau)\diamond  A_{1}\oplus_{0}\tau \diamond  A_{2}$ is $C$-coconvex, one
has
\begin{eqnarray}
\frac{ \overline{V}_0((1-\tau)\diamond  A_{1}\oplus_{0}\tau\diamond  A_{2}, A_1)}{V_n((1-\tau)\diamond  A_{1}\oplus_{0}\tau\diamond   A_{2})}
&\leq& \frac{1}{n}\cdot \log\bigg(\dfrac{V_{n}(A_{1})}{V_{n}((1-\tau)\diamond  A_{1}\oplus_{0}\tau\diamond   A_{2})}\bigg),
\label{two-inq0log0} \\
\frac{\overline{V}_0((1-\tau)\diamond  A_{1}\oplus_{0}\tau\diamond  A_{2}, A_2)}{V_n((1-\tau)\diamond  A_{1}\oplus_{0}\tau\diamond  A_{2})}
&\leq& \frac{1}{n}\cdot \log\bigg(\dfrac{V_{n}(A_{2})}{V_{n}((1-\tau)\diamond  A_{1}\oplus_{0}\tau\diamond  A_{2})}\bigg).
\label{two-inq-log}
 \end{eqnarray}
 By \eqref{support-def-log} and \eqref{def-log-mix}, one can get
$$0= (1-\tau) \overline{V}_0((1-\tau)\diamond  A_{1}\oplus_{0}\tau\diamond  A_{2},
A_1)+\tau \overline{V}_0((1-\tau)\diamond  A_{1}\oplus_{0}\tau \diamond  A_{2}, A_2).$$
Together with \eqref{two-inq0log0} and \eqref{two-inq-log}, one gets
\begin{eqnarray} 0\leq (1-\tau)
\log\Big(\dfrac{V_{n}(A_{1})}{V_{n}((1-\tau)\diamond  A_{1}\oplus_{0}\tau\diamond
A_{2})}\Big) +\tau
\log\Big(\dfrac{V_{n}(A_{2})}{V_{n}((1-\tau)\diamond  A_{1}\oplus_{0}\tau\diamond
A_{2})}\Big),\label{equ-cha-1-1}  \end{eqnarray} which is exactly
the desired inequality  \eqref{1.5-6} after simplification. To
characterize the equality for \eqref{1.5-6}, without loss of
generality, we can assume that $V_n(A_1)=V_n(A_2)=1$ (due to the
homogeneity of  \eqref{1.5-6} for $A_1$ and $A_2$). In this case, if
equality holds in inequality \eqref{1.5-6}, one must have equality
in \eqref{equ-cha-1-1} and hence in  \eqref{two-inq0log0} and
\eqref{two-inq-log}. It follows from the characterization for
equality in \eqref{log-ine-1-1-6} that both $A_1, A_2$ are dilations
of $(1-\tau)\diamond  A_{1}\oplus_{0}\tau\diamond   A_{2}$. This further implies that
equality holds in \eqref{1.5-6} only if  $A_1=\alpha A_2$ for some
$\alpha>0$.  Conversely,  it is trivial to check that the equality
holds in \eqref{1.5-6} if $A_1=\alpha A_2$ for some $\alpha>0$ and
this completes the characterization of the equality of the
log-Brunn-Minkowski inequality \eqref{1.5-6}.
\end{proof}

In \cite[Lemma 9]{Schneider2018}, the cone-volume measure of a
$C$-close set $A^{\bullet}$, denoted by $V_{A^{\bullet}}$, on $\Omega_C$ has the following
representation: for any Borel sets $\eta\subset\Omega_C$,
\begin{equation}\label{def-volume-measure}
V_{A^{\bullet}}(\eta)=\frac{1}{n}\int_{\eta} -h(A^{\bullet}, u)\,d S_{n-1}(A^{\bullet}, u).\end{equation}
Again let $ \overline{V}_{C\setminus A^{\bullet}} =V_{A^{\bullet}}$ be the cone-volume
measure of $C\setminus A^{\bullet}$.  Clearly, $V_{A^{\bullet}}(\Omega_C)=V_n(C\setminus
A^{\bullet})<\infty$ as $C\setminus A^{\bullet}$ is a $C$-coconvex set. Hence the
cone-volume measure $V_{A^{\bullet}}$ is finite on $\Omega_C$.   The problem to
characterize the cone-volume measure of $C$-close sets has been
proposed by Schneider in \cite{Schneider2018}. This problem may be
called the $L_0$ (or log) Minkowski problem.
\begin{problem}[Log-Minkowski problem]\label{log-mink-1} Under what
necessary and/or  sufficient conditions on a nonzero finite Borel
measure $\mu$ on $\Omega_C$ does there exist a $C$-close set
$A^\bullet$ with $A=C\setminus A^\bullet$ such that $\mu= \overline
{V}_A$?
\end{problem}
The existence of solutions to the log-Minkowski problem has been
provided in \cite[Theorems 4 and 5]{Schneider2018}. The following
result confirms that the solutions to the log-Minkowski problem are
indeed unique and hence solves the open problem raised by  Schneider
in \cite[p. 203]{Schneider2018}.

\bt \label{14-log} Let $A_1^{\bullet}, A_2^{\bullet}$ be two
$C$-close sets such that $A_1=C\setminus A_1^{\bullet}$ and
$A_2=C\setminus A_2^{\bullet}$ are two $C$-coconvex sets.  Then
$A_{1}=A_{2}$, if the following identity holds on $\Omega_C$:
\begin{equation}\label{assumption-1-log} \overline{V}_{A_{1}} =
\overline{V}_{A_{2}}. \end{equation} In particular, the solutions to
the log-Minkowski problem (i.e., Problem \ref{log-mink-1}) are
unique.
 \et
\begin{proof}
It follows from \eqref{def-log-mix}, \eqref{def-volume-measure} and
\eqref{assumption-1-log} that  $V_n(A_1)=V_n(A_2)$ and
\begin{eqnarray}  \overline{V}_0(A_1,
A_2)&=&\dfrac{1}{n}\int_{\Omega_C} \log\bigg(\frac{\overline{h}(A_2,
u)}{\overline{h}\left(A_1, u\right)}\bigg) \overline{h}\left(A_1,
u\right)\,d\overline{S}_{n-1}\left(A_1, u\right) \nonumber \\
&=&\int_{\Omega_C} \log\bigg(\frac{\overline{h}(A_2,
u)}{\overline{h}\left(A_1, u\right)}\bigg) \,d
\overline{V}_{A_1}\left(u\right)\nonumber \\
&=&\int_{\Omega_C} \log\bigg(\frac{\overline{h}(A_2,
u)}{\overline{h}\left(A_1, u\right)}\bigg) \,d
\overline{V}_{A_2}\left(u\right) \nonumber \\
&=&-\dfrac{1}{n}\int_{\Omega_C} \log\bigg(\frac{\overline{h}(A_1,
u)}{\overline{h}\left(A_2, u\right)}\bigg)  \overline{h}\left(A_2,
u\right)\,d\overline{S}_{n-1}\left(A_2, u\right)\nonumber \\
&=&-\overline{V}_0(A_2, A_1). \label{two-neg-log}\end{eqnarray} An
application of the  log-Minkowski inequality \eqref{log-ine-1-1-6},
together with the fact that $V_n(A_1)=V_n(A_2)$, yields
\begin{eqnarray}
\frac{ \overline{V}_0(A_1, A_2)}{V_n(A_1)}\leq \frac{1}{n}\cdot
\log\bigg(\dfrac{V_{n}(A_{2})}{V_{n}(A_{1})}\bigg)=0
\ \ \mathrm{and} \ \ \frac{ \overline{V}_0(A_2, A_1)}{V_n(A_2)}
\leq \frac{1}{n}\cdot \log\bigg(\dfrac{V_{n}(A_{1})}{V_{n}(A_{2})}\bigg)=0.\label{two-zeros-log}
 \end{eqnarray}
It can be easily checked by \eqref{two-neg-log} and
\eqref{two-zeros-log}  that
\begin{eqnarray*}
\frac{ \overline{V}_0(A_1, A_2)}{V_n(A_1)}
= \frac{1}{n}\cdot \log\bigg(\dfrac{V_{n}(A_{2})}{V_{n}(A_{1})}\bigg)=0.
 \end{eqnarray*}
By the characterization of equality for  the  log-Minkowski
inequality \eqref{log-ine-1-1-6}, one gets $A_1=\alpha A_2$ for some
$\alpha>0$. Moreover,  $V_n(A_2)=V_n(A_1)=V_n(\alpha A_2)=\alpha^n
V_n(A_2),$ which implies $\alpha=1$. Hence $A_1=A_2$ as desired.
\end{proof}

\vskip 2mm \noindent  {\bf Acknowledgement.} The research of JY is supported by NSFC (No.\ 11971005).  The research of DY is
supported by an NSERC grant, Canada. The research of BZ is supported by NSFC (No.\ 11971005) and the Fundamental Research Funds for the Central Universities (No. GK202102012). The authors are greatly indebted to the reviewers for many valuable comments which greatly improves the quality of the paper.

\vskip 2mm \noindent Jin Yang, \ \ \ {\small \tt yangjin95@126.com}\\
{ \em School of Mathematics and Statistics, Hubei Minzu University, Enshi, 445000, China}\\
{ \em School of Mathematics, Sichuan University, Chengdu, Sichuan, 610000, China }

\vskip 2mm \noindent Deping Ye, \ \ \ {\small \tt deping.ye@mun.ca}\\
{\em Department of Mathematics and Statistics, Memorial
University of Newfoundland, St. John's, Newfoundland A1C 5S7,
Canada
}

\vskip 2mm \noindent Baocheng Zhu, \ \ \ {\small \tt zhubaocheng814@163.com}\\
{ \em School of Mathematics and Statistics, Shaanxi Normal
University, Xi'an, 710062, China}

\end{document}